\documentclass[10pt]{article}
\usepackage[total={7.2in, 9in}]{geometry}
\usepackage[colorlinks,citecolor=blue,urlcolor=blue]{hyperref}
\usepackage[utf8]{inputenc}
\usepackage[english]{babel}
\usepackage{subcaption}
\usepackage{hyperref}
\usepackage{breakcites}
\usepackage{hyperref}
\usepackage{graphicx}
\usepackage{amsthm,amsmath,amssymb}
\usepackage{mathrsfs} 
\usepackage{float}
\usepackage{mathtools}
\usepackage{comment}
\usepackage{xcolor}
\usepackage{xurl}
\usepackage{authblk}
\usepackage{enumitem} 
\usepackage{xcolor}
\usepackage{parskip}
\usepackage[capitalise,noabbrev]{cleveref}

\theoremstyle{definition}
\newtheorem{theorem}{Theorem}
\newtheorem{lemma}{Lemma}[section]
\newtheorem{proposition}[lemma]{Proposition}

\newtheorem{rem}[lemma]{Remark}
\newtheorem{cor}[lemma]{Corollary}
\newtheorem{definition}[lemma]{Definition}

\newcommand\vep{\varepsilon}

\newcommand{\Nc}{\mathcal{N}}
\newcommand{\pp}{\mathbb{P}}

\newcommand{\EE}[1]{\mathbb{E} \left[ #1 \right]}
\renewcommand{\P}{\mathbb P}

\newcommand{\E}{\mathbb E}
\newcommand{\R}{\mathbb R}

\newcommand{\N}{\mathbb N}

\newcommand{\lc}{{L(\beta)}}
\newcommand{\lcn}{{L}}
\newcommand{\Dc}{[0,\lc]}
\newcommand{\Dcn}{{\Omega}}

\usepackage{tocloft}

\title{Power-law scaling of the effective population size in a branching particle system
for moderate mutation-selection}

\date{\today}

\begin{document}

\author{Florin Boenkost\thanks{Faculty of Mathematics, University of Vienna, Oskar-Morgenstern-Platz 1, 1090 Vienna, Austria} and Julie Tourniaire\thanks{Université Marie et Louis Pasteur, CNRS, LmB (UMR 6623), F-25000 Besançon, France}}

\maketitle
\abstract{
We consider a one-dimensional dyadic branching
Brownian motion on $\R$ with positive drift $\beta \in (0,1)$, branching rate $1/2$,
reflected at $0$ and
killed at a boundary $L > 0$. The killing boundary
$L$ is chosen so that the total population size remains approximately
constant, proportional to $N \in \mathbb{N}$. This branching process models a
population accumulating deleterious mutations.
In the large-$N$ limit, we prove that when the typical width of the particle cloud
is of order $c \log(N)$, with $c \in (0,1)$, the demographic fluctuations
follow a Yaglom law on a polynomial time scale. Moreover, the limiting
genealogy  of the system involves only binary mergers, concentrated near the
reflecting boundary. Our model is a version of the branching Brownian motion
with absorption introduced by Berestycki, Berestycki, and Schweinsberg
to study the effect of beneficial mutations on genealogies.
In sharp contrast with their model, whose genealogy is given by a Bolthausen–Sznitman
coalescent, we show that our system falls into the universality class of
Kingman’s coalescent.

\section{Introduction}

Under neutral theory, genetic diversity is expected to increase with the census population size $N$. Empirically, however, genetic diversity typically scales with $N$ according to a power law~\cite{buffalo}. This discrepancy is known as Lewontin's paradox~\cite{Lewontin1974}. Several explanations have been proposed to account for this phenomenon; we refer to~\cite{buffalo} for a detailed review.
Negative selection, that is, selection against new mutations, is one such explanation.
It is ubiquitous in nature and known to reduce genetic diversity, thereby motivating
extensive research (see, e.g.,~\cite{Walczak2012,cvijovic2018effect,etheridge_2011}).
Despite these efforts, the mechanisms underlying this reduction in genetic diversity remain only partially understood.

In this work, we introduce a branching particle system as a tractable model for negative selection in a biologically relevant parameter regime. Within this framework, we find that the evolutionary dynamics of the system closely resemble those of a neutral population with a polynomially reduced \textit{effective population size}, in agreement with predictions by Charlesworth, Morgan, and Charlesworth~\cite{Charlesworth93}.
Interestingly enough, the exponent in the power law can take any value, provided that the strength of selection is adjusted accordingly.
As a byproduct, our approach also provides a phenomenological perspective on how negative selection shapes population genealogies.

In this introductory section, we define our model and the moderate mutation-selection
regime (Section~\ref{sec:model}).
In Section~\ref{sec:motivation}, we provide a biological interpretation
of this model as a system of particles accumulating deleterious mutations.
Our main results concerning the demographic fluctuations and genealogical
structure of the system are presented in Section~\ref{sec:results}.

\subsection{The model}\label{sec:model}
\label{sec:BBM}
Let  $\beta\in(0,1)$ and define
\begin{equation}
	L(\beta):=\frac{1}{\sqrt{1-\beta^2}}\left(\arctan\left(-\frac{\sqrt{1-\beta^2}}{\beta}\right)+\pi\right).
	\label{eq:def_L}
\end{equation}
Consider a dyadic branching Brownian motion $(\mathbf{X}^\beta_t)_{t>0}$
(BBM) with drift $\beta$, branching rate $1/2$,
reflection at $0$, and killing at $L(\beta)$.

Let $\mathcal N_t^\beta$ denote the set of particles in the system at time $t$ and for all $v \in \mathcal N^\beta_t$,  let $x^\beta_v=x^\beta_v(t) \in [0,\lc)$ denote the position of particle $v$. Let $Z^\beta(t):=|{\mathcal N}^\beta_t|$ be the number of particles in the system at time $t$.
We write $\mathbb{P}_x$ for the law of the BBM started from a single particle at $x\in[0,\lc)$ and $\mathbb{E}_x$ for the corresponding expectation. The natural filtration generated by the BBM is denoted by $(\mathcal{F}^\beta_t,t\geq 0)$.

The killing boundary $\lc$ is chosen so that the number of particles in the
BBM remains approximately constant, making the system critical in a certain sense.
In the next paragraph, we explain why \eqref{eq:def_L} is the relevant quantity to consider.

\paragraph{Criticality.} It is a standard result (see e.g.~\cite[p.188]{Lawler:2018vn}) that the expected number of particles in the BBM $\mathbf{X}^\beta$ is governed by the PDE
\begin{equation}\label{PDE:A}\tag{A}
	\begin{cases}
		\partial_tu(t,y)=\frac{1}{2}\partial_{yy}u(t,y)-\beta\partial_y u(t,y)+\frac{1}{2}u(t,y) \\
		\beta u(t,0)-\frac{1}{2}\partial_y u(t,y)|_{y=0}=0 & \text{(flux at $0$),}               \\
		u(t,\lc)=0                                         & \text{(killing at $\lc$)}.
	\end{cases}
\end{equation}
Let $p^\beta_t(x,y)$ denote the fundamental solution of \eqref{PDE:A}. Then
$p^\beta_t$ represents the \emph{density} of particles in the BBM in the following sense.
\begin{lemma}[Many-to-one lemma]\label{lem:many-to-one0} For every measurable function $f: \R \mapsto \R$, for all $x \in [0,\lc]$ and $t \geq 0$, we have
	\begin{align}
		\E_x \left[ \sum_{v \in \Nc_t^\beta} f(x_v)\right] = \int_{0}^{L(\beta)} f(y) p^\beta_t(x,y)dy.
	\end{align}
\end{lemma}
\noindent In particular, note that, if the system starts with a single particle at $x$ at time $0$,   the expected number of particles in any Borel set $B$ at time $t$ is given by $\int_B p^\beta_t(x,y)dy$. Define
\begin{equation}\label{def:q}
	g^\beta_t(x,y): =e^{\beta(x-y)}e^{\frac{\beta^2-1}{2}t}p^\beta_t(x,y).
\end{equation}
A straightforward computation shows that $g^\beta_t$ is the fundamental solution of the self-adjoint PDE
\begin{equation}\label{PDE:B}\tag{B}
	\begin{cases}
		\partial_t u(t,y)=\frac{1}{2}\partial_{yy}u(t,y) &                            \\
		\beta u(t,0)-\partial_y u(t,y)_{|y=0}=0          & \text{(flux at $0$)}       \\
		u(t,\lc)=0                                       & \text{(killing at $\lc$)}.
	\end{cases}
\end{equation}
For $\beta \in(0,1)$ and $L>0$, consider the Sturm--Liouville problem
\begin{equation}\tag{SLP}
	\frac{1}{2}v''(x)=\lambda v(x), \quad x\in(0,L), \quad v'(0)=\beta v(0), \quad v(L)=0.
	\label{SLP}
\end{equation}
It is known \cite[Chapter 4]{zettl10} that the eigenvalues $\lambda_i\equiv \lambda_{i}^{\beta,L}$ of \eqref{SLP}
are simple and that the associated eigenvectors $(v_i)_{i\geq 1}\equiv (v_{i}^{\beta, L})_{i\geq 1}$  can be normalised to form an orthonormal basis of $\mathrm{L}^2([0,L])$. A direct calculation shows that the $i$-th eigenvalue of \eqref{SLP} is the unique solution of the algebraic equation
\begin{equation} \label{eq:caract_eigenvalue}
	\tan\left(\sqrt{-2\lambda_i}L\right)=-\frac{\sqrt{-2\lambda_i}}{\beta},
	\quad \sqrt{-2\lambda_i}L\in\left[\left(i-\tfrac{1}{2}\right)\pi,i\pi\right].
\end{equation}
For sufficiently large $t$, we expect the density of particles in the dyadic BBM with drift $\beta$, branching at rate $1/2$, reflected at $0$ and killed at $L$, to be well-approximated by the first term of its spectral decomposition, that is
\begin{equation}\label{approx:p}
	t\gg 1, \quad
	p_t(x,y) \approx e^{\beta(y-x)}e^{\frac{1-\beta^2}{2}t} e^{\lambda_1 t} \frac{v_{1}(x)v_1(y)}{||v_1||^2}.
\end{equation}
One can now choose $L$ in such a way that the expected number of particles neither increases nor decreases exponentially: for this choice of $L$, we say that the BBM is critical. This motivates our definition for $\lc$.
Indeed, for  $L=\lc$, one can show (see \eqref{eq:caract_eigenvalue})
that $\lambda_1=-\frac{1-\beta^2}{2}$ so that $\mathbf{X}^\beta$ is critical
and  check that
\begin{equation}
	\label{eq:def_gamma}
	v_{1}^\beta (x)\equiv v_{1}^{\beta,\lc}(x)=\sin(\gamma(\beta)(\lc-x)), \quad \text{with} \quad \gamma(\beta):=\sqrt{1-\beta^2}
\end{equation}
is a positive solution of  \eqref{SLP} for $\lambda=-\frac{1-\beta^2}{2}$.
\begin{rem}\label{rem:bij_L}
	The function $[0,1)\to[\pi/2,\infty), \; \beta\mapsto \lc$ is increasing, $L_0=\pi/2$ and $\lc\to\infty$ as $\beta\to 1$. Moreover,
	\begin{equation}
		\label{eq:equiv_L}
		\gamma(\beta) \lc= \pi-\gamma(\beta)+o(\gamma(\beta)), \quad \beta\to 1.
	\end{equation}
\end{rem}

\paragraph{Perron-Frobenius type decomposition.} We define
\begin{equation}\label{def:h}
	\tilde{h}^\beta(x)=\tilde{c}(\beta) e^{\beta x} v_{1}^{\beta}(x), \quad \text{and} \quad h^\beta(x)=\frac{1}{\tilde{c}(\beta) }\frac{2}{\lc+\beta}e^{-\beta x}v_{1}^{\beta}(x), \quad x\in\Dc,
\end{equation}
with
\begin{equation*}
	\tilde{c}(\beta) :=\left(\int_0^{\lc} \tilde h^\beta(x)dx\right)^{-1}=\frac{1}{\gamma(\beta)}e^{-\beta \lc}.
\end{equation*}
The function ${h}^\beta$ (resp.~$\tilde{h}^\beta$) is a right (resp.~left) Perron-Frobenius
eigenvector of the differential operator $\frac{1}{2}\partial_{xx}+\beta\partial_x+\frac{1}{2}$  with the same boundary conditions as in \eqref{PDE:A} and $\tilde{c}(\beta)$ can be thought of as a Perron-Frobenius renormalisation constant. Indeed, we have
\begin{equation}\label{eq:normalisation}
	\int_0^{\lc}\tilde h^\beta (x)dx=1, \quad \text{and} \quad \int_0^\lc h^\beta(x)\tilde{h}^\beta(x)dx=1.
\end{equation}
With this notation, \eqref{approx:p} shows that the density $p^\beta_t(x,y)$ is
approximately given by $h^\beta(x)\tilde{h}^\beta(y)$.
Furthermore,  \cref{lem:many-to-one0} and \eqref{eq:normalisation} imply  that,
for sufficiently large $t$,
\begin{equation}
	\label{eq:reproductive_value}
	\E_x[Z^\beta(t)]\approx h^\beta(x),
\end{equation} }
so that $h^\beta(x)$ can be thought of as the \textit{reproductive value} of a particle located at $x$.
On the other hand, the probability distribution $\tilde{h}^\beta(y)dy$
may be interpreted as the \emph{stable configuration} of the system: \cref{lem:many-to-one0} and \eqref{eq:normalisation}  show that, when starting from $N$ particles distributed according
to $\tilde h$, the expected configuration of the system remains essentially
the same for all sufficiently large $t$.

{ Next,
following~\cite{tourniaire2023tree}, we define
\begin{equation}
	\label{eq:def_Sigma}
	(\Sigma^\beta)^2=\int_0^Lh^\beta(x)^2\tilde h^\beta(x) dx.
\end{equation}
From the interpretation of $h$ and $\tilde h$, the quantity $(\Sigma^\beta)^2$ can be
viewed as a measure of the reproductive variance in the system, taking its structure into account.

\paragraph{The moderate mutation-selection regime.}
\label{sec:weakselection_BBM}

We are interested in the limiting behaviour of a sequence of critical branching
Brownian motions. More precisely, we set $c\in(0,1)$ and consider a sequence of drifts $(\beta_N)_{N\geq N_0}$ such that
\begin{equation}\label{eq:def_beta}
	\forall N\geq N_0, \quad  L(\beta_N)=c\log(N)+6\log \log(N),
\end{equation}
where $N_0:=\min\{n\geq 2:c\log(n)+\log\log(n)>\pi/2\}$. For all $N\geq N_0$, we write $(\mathbf{X}^{\beta_N}(t))\equiv (\mathbf{X}^{N}(t))$ for the corresponding critical BBM with drift $\beta_N$.
For notational convenience, we simplify the notation for quantities related to the BBM $\mathbf{X}^{N}$
by writing
\begin{equation}
	L(\beta_N)\equiv L,\quad \gamma(\beta_N)=\gamma,  \quad Z^{\beta_N}(t)=Z(t),\quad  h^{\beta_N}\equiv h,  \quad \tilde{h}^{\beta_N}\equiv \tilde{h},
	\quad \Sigma^{\beta_N}\equiv \Sigma,
\end{equation}
so that, throughout the paper, these quantities are understood to depend implicitly on $N$.
The choice of parameter \eqref{eq:def_beta}, which we refer to as the
weak-selection regime in the BBM with absorption and reflection,
is motivated by the following two observations.
\begin{proposition}\label{prop:concentration} 	Let $(\beta_N)$ be as in \eqref{eq:def_beta}.
	There exists a positive constant $\sigma$ that only  depends on $c$, such that
	\begin{equation}
		\label{eq:equiv_variance}
		\lim_{N\to\infty}\frac{\Sigma^2}{N^c}=\sigma^2.
	\end{equation}
\end{proposition}
The second observation is that the main contributions to
the integral~\eqref{eq:def_Sigma} are concentrated near the
reflective boundary.
\begin{proposition}\label{prop:concentration2}
	Let $(\beta_N)$ be as in \eqref{eq:def_beta}. For all $1\ll z\ll L$,
	\[
		\lim_{N\to\infty}\frac{1}{N^c}\int_0^z h(x)^2\tilde h(x)dx=\sigma^2,
	\]
	Fix ${A}=6\log\log N-\log\log\log(N)$. Then
	\[
		\lim_{N\to\infty} N^c\int_0^{{A}} \tilde{h}(x)dx =a,
	\]
	for some constant $a>0$ (that only depends on $c$).
\end{proposition}

Proposition~\ref{prop:concentration} and Proposition~\ref{prop:concentration2} suggest that the demographic fluctuations of the BBM should coincide with those of a
subsystem consisting of particles that remain in the interval $[0,A]$.
The second part of the proposition shows that the size of this subsystem can be
\textit{reduced} to order $N^{1-c}\ll N$.
The main goal of the present article is to establish that, for such a structured population,
in the large-$N$ limit, the evolutionary dynamics resemble those of a single-type critical
Galton--Watson process with finite reproductive variance and population size of order
$N^{1-c}$.

We postpone the proof of Proposition~\ref{prop:concentration} and Proposition~\ref{prop:concentration2}
to Appendix~\ref{sec:proof_variance}.

\subsection{Main results} \label{sec:results}

In this paper, we study the limiting genealogy and demographic fluctuations of
the BBM in the large-$N$ limit, when the process starts
with a single particle at a fixed location $x > 0$.
As a first step, we derive precise asymptotics for the survival probability of the sequence of BBMs $(\mathbf{X}^N, \; N\geq N_0)$.
\begin{theorem}[Kolmogorov estimate]\label{th:Kolmogorov}
	Let $t>0$ and $x>0$. We have
	$$
		\left|{N}\mathbb{P}_x\left( Z(tN^{1-c})>0 \right) - \frac{2}{\sigma^2 t} h(x)\right|\to 0, \quad \text{as } N \to \infty,
	$$
	where $\sigma^2$ is as in Proposition~\ref{prop:concentration}.
\end{theorem}
This theorem suggests that the typical time scale of evolution of  $Z^N$ is of order $N^{1-c}$. The following Yaglom law describes the scaling limit of $Z^N$ on this same time scale.
\begin{theorem}[Yaglom law]\label{th:Yaglom}
	Let $t>0$ and $x>0$. Assume that, for all $N$ large enough,  $\mathbf{X}^N_0$ consists of a single
	particle at $x$. Conditional on the event $\{Z^N(tN^{1-c}) > 0\},$ we have
	\begin{equation*}
		\frac{1}{N}{Z(tN^{1-c})}
		\to
		\frac{\sigma^2 t}{2} \ \mathcal{E}, \quad \text{as } N \to \infty,
	\end{equation*}
	in distribution, where $\mathcal{E}$ is an exponential random variable of mean $1$.
\end{theorem}
We now turn to the description of the genealogy of the system.  For two particles
$v_1,v_2\in\mathcal{N}^N_t$, we denote by $d_t^N(v_1,v_2)$ the time to the most recent
common ancestor (MRCA) of $v_1$ and $v_2$.
Intuitively, our next result shows that the rescaled distance matrix of $k$
individuals sampled in the BBM converges to the distance matrix of $k$ individuals
sampled from a critical Galton-Watson process \cite{popovic2004}.
This limiting object is defined as follows. Let $t>0$ and $U$ be a random
variable on $[0,t]$.  Define $U^\theta$ such that
\begin{align} \label{eq:definition_H_theta}
	\forall s \le t, \qquad	\P( U^{\theta} \le s )
	:= \frac{(1+\theta) \P(U \le s)}{1+\theta \P(U \le s)}
	= \frac{(1+\theta) (s/t)}{1+\theta (s/t)}.
\end{align}
Let $(U^{\theta}_{i}; i\in\{1,...,k\})$ be $k$ i.i.d. copies of $U^\theta$ and set
$$
	\forall 1\leq i< j\leq k, \quad U^{\theta}_{i,j} = U^{\theta}_{j,i} :=  \max\{U^\theta_{l}: l\in\{i,\cdots,j-1\}\}.
$$
Define the random distance matrix $(H_{i,j}):= (H_{i,j}; i\neq j \in[k])$ such that for every  bounded and continuous function $\phi :\R^{k^2}\to \R$,
\begin{align} \label{eq:moments-CPP}
	\E\big[ \phi\big( (H_{i,j}) \big) \big]
	= k\int_0^\infty \frac{1}{(1+\theta)^2}
	\Big(\frac{\theta}{1+\theta}\Big)^{k-1}
	\E\big[\phi\big( (U^{\theta}_{i,j}) \big)\big]
	d\theta.
\end{align}
\begin{theorem}[Genealogy]\label{th:genealogy}
	Let $t>0$ and $x\in\mathbb{R}$. Assume that, for all $N$ large enough,  $\mathbf{X}^N_0$ consists of a single
	particle at $x$ and condition the BBM $\mathbf{X}^N$ on the event $\{Z^N(tN^{1-c})>0\}$. Sample $k$ particles $(v_1,...,v_k)$ in $\mathbf{X}^N$ at time $tN^{1-c}$. The rescaled distance matrix $\left(\frac{1}{N^{1-c}}d^N_{tN^{1-c}}(v_i,v_j)_{i,j}\right)$ converges to $(H_{\sigma(i),\sigma(j)})$ in distribution, where $H$ is as in \eqref{eq:moments-CPP} and $\sigma$ is an independent permutation of $\{1,...,k\}$.
\end{theorem}
\begin{rem}
	{The limiting binary genealogy can be obtained on any time scale $N^{1-c}$, with $c\in(0,1),$ by choosing $L$ in \eqref{eq:def_beta} accordingly. }
\end{rem}

\begin{rem}
	\label{rk:Feller}
	Alternatively, one could describe this scaling limit when the process starts with approximately $N$ particles in a {front-like} configuration, i.e.~distributed
	according to the stable configuration $\tilde h$. In this setting,
	the system should exhibit a ``fitness wave” behaviour: the population size
	stays of order $N$ and fluctuate around the stable configuration $\tilde h$. After rescaling the total population size by
	$N$, fluctuations of order $1$ in the process $Z$ are
	expected to emerge on timescale $N^{1-c}$.
	Our Yaglom law suggests that these demographic fluctuations converge to a Feller diffusion.
	The proof of this fact can be obtained by combining our Kolmogorov estimate and
	our Yaglom law. The key idea of the proof is to control the Laplace transform
	of the additive martingale
	\begin{equation}
		Y(t)=\sum_{v\in \mathcal{N}_t}h(x_v(t)),
		\label{eq:def_martingale}
	\end{equation}
	as done in \cite{tourniaire2021branching}.
	In turn, this Laplace transform can be controlled by computing all moments of the additive martingale
	$Y$, following the same line of argument as developed in this paper.

	It is known \cite{birkner2005alpha}
	that the genealogy of a population whose demographic fluctuations are described
	by Feller diffusion is given by a time-changed Kingman coalescent.
	In the present work, we show
	that, on a deterministic timescale, the genealogical structure spanned by $k$ individuals sampled at a fixed
	time horizon only comprises binary mergers (see Theorem~\ref{th:genealogy}).

\end{rem}

\subsection{Biological motivation}

\label{sec:motivation}

The Wright-Fisher (WF) model with mutations and selection
is a classical framework for studying the evolutionary dynamics of an asexual population
undergoing selection. The population size is constant, equal
to $N\in\mathbb{N}$, and the system evolves in discrete time. The WF dynamics depend
on two parameters: the mutation rate $\mu_N>0$ and the selection coefficient
$s_N\in\mathbb{R}$.
At each step, the population is renewed in two steps. First,
each of the $N$ new individuals independently chooses a parent from the previous
generation. The probability that a parent of type $k$, i.e.,~carrying $k$ mutations,
is chosen is
proportional to its fitness $(1+s_N)^k$.  Each offspring inherits the mutations of its parent.
Second, the $i$-th individual acquires $X_i$ new mutations, where
$(X_i)_{1\leq i\leq N}$ is a sequence of i.i.d.~Poisson($\mu_N$)-random variables.
Mutations are said to be deleterious (resp.~beneficial)
or equivalently, that selection is negative (resp.~positive) if $s_N<0$ (resp.~$s_N>0$).

The evolutionary dynamics of an asexual population accumulating beneficial
mutations are now well
understood~\cite{neher2013genealogies,berestycki13,durrett2011,schweinsberg2017rigorous,roberts2021gaussian}.
In this setting, selection favours new mutations,
which can rapidly spread and fix in the population.
Looking backward in time, the genealogy is described by a multiple-merger
coalescent, capturing the fact that many lineages trace their ancestry
to a few highly successful individuals.
As a result, populations undergoing positive selection typically exhibit reduced genetic diversity compared to neutral
populations ($s_N=0$).
When selection is negative,
the population undergoes two opposite forces: \textit{natural selection},
which tends to eliminate individuals carrying many mutations, and \textit{mutational pressure},
which introduces new deleterious mutations.
This interplay leads to complex evolutionary dynamics that are not yet fully understood.
In this setting, the type with the fewest mutations is called the \emph{best class}.
Because mutations are unidirectional and the population is finite,
the current best class will eventually be lost.
This phenomenon is referred to as \emph{Muller's ratchet}~\cite{Muller1964}.
Many studies have investigated this model to estimate the rate at which the ratchet clicks~\cite{Etheridge2009,Igelbrink2023,Casanova2022,Pfaffelhuber_2012}.
It is also known since the work of Haigh~\cite{Haigh1978} that between two clicks, the distribution of types (often referred to as the fitness wave) is
well-approximated by a Poisson distribution with parameter $\mu_N/s_N$.

In this framework, the family of parameter scalings
\begin{equation}
	\label{eq:log_bulk}
	\mu_N=N^{-\beta}, \quad \frac{\mu_N}{s_N}=-c\log (N), \quad \beta\in(0,1), \;
	c\in (0,1),
\end{equation}
is commonly referred to as the moderate mutation-selection regime~\cite{Igelbrink2023}. For this choice of parameters, we see that
the bulk of the fitness wave is concentrated around $c\log(N)$, and
the best class contains, in expectation, $N^{1-c}$ individuals.
The BBM $\mathbf{X}^N$ is constructed in such a way that the macroscopic behaviour of the branching system
resembles that of the Wright-Fisher dynamics in the moderate mutation-selection regime:
the bulk of the type distribution $\tilde h$ is within a distance of
order $1$ of $c \log(N)$ for some $c\in(0,1)$
(see \eqref{eq:def_beta} and \eqref{eq:normalisation}).

We now briefly interpret the result stated in
Propositions~\ref{prop:concentration} and~\ref{prop:concentration2}.
As discussed in Section~\ref{sec:BBM} $h(x)$ can be thought of as the reproductive
value of a particle located at $x$,  the distribution $\tilde{h}(y)\,dy$ describes the stable
configuration of the system and $\Sigma^2$ measures the {reproductive variance} of the BBM.
Propositions~\ref{prop:concentration} and~\ref{prop:concentration2} show that the integral defining
$\Sigma^2$ is concentrated in $[0, A]$, suggesting that the demographic
fluctuations (and thus the evolutionary dynamics) of the BBM are driven by particles located in this region. In $[0, A]$,
the expected number of individuals is of order $N^{1-c}$ (see Proposition~\ref{prop:concentration2}),
matching the size of the best class in the WF model
under the Poissonian approximation.

Let us now discuss the content of \cref{th:Kolmogorov}, \cref{th:Yaglom}, and \cref{th:genealogy}. These results
suggest that the evolutionary dynamics of the fitness wave resemble those of a neutral population~\cite{Kingman1982} on a polynomial timescale $N_e:=N^{1-c}$.
In population genetics, the quantity $N_e$ corresponds to the \textit{effective population size}~\cite{wright1931evolution,charlesworth2009effective} of the
system, that is,  the size
of an idealised unstructured neutral population that would experience
the same level of genetic drift as the actual population under study.
Our results indicate that the evolutionary
dynamics of a population accumulating deleterious mutations is driven by
its best class, in agreement with previous conjectures~\cite{Charlesworth93,BKT2023+}.

\label{sec:related}
\subsection{Connection with previous work}
Analyzing the Wright–Fisher dynamics introduced in Section \ref{sec:motivation}
is challenging, as constant population size induces strong dependencies between
individuals. To address this, diffusion approximations have been proposed~\cite{Etheridge2009,Pfaffelhuber_2012},
whereas other works have tried to simplify the discrete model \cite{Igelbrink2023,Casanova2022}.
We follow the latter approach and replace natural selection with \textit{truncation selection},
removing individuals with insufficient fitness. We further assume the accumulation of
many small-effect mutations, so that the evolution of fitness can be approximated by
Brownian motion \cite{Desai:2007aa,Barton:2017aa}.

Our model is a variant of the BBM with absorption studied by Berestycki, Berestycki,
and Schweinsberg~\cite{berestycki13}, which was introduced to investigate the
effect of the accumulation of beneficial mutations on the genealogy of a population.
In~\cite{berestycki13}, the authors consider a BBM with drift $-1$ and branching
rate $1/2$, where particles are killed upon reaching $0$. In this framework,
each individual in the population is represented by a point $x \in \R^+$,
interpreted as its fitness.
Fitness evolves as a Brownian motion, while selective pressure is represented by
a moving barrier with constant speed $1$, which removes individuals whose fitness
becomes too low.
Intuitively, a particle located far to the right corresponds to an individual
that has accumulated many beneficial mutations, more rapidly than the rest of the
population. Such an individual produces a large number of offspring before eventually
being absorbed by the moving barrier. In this setting,
the genealogy of the system converges to the Bolthausen--Sznitman
coalescent~\cite{Bolthausen1998}, and fluctuations in the population size
are described by Neveu's continuous-state branching process~\cite{neveu}.

In the present work, we model individuals accumulating deleterious mutations:
the fitness of a particle is now given by $L - x$.
While in~\cite{berestycki13} particles are allowed to accumulate an arbitrarily large
number of beneficial mutations, here, the fittest individuals are those with
the fewest mutations, which cannot be less than $0$. In our BBM model, this lower
bound on the number of
mutations is enforced through reflection at $0$. This reflective boundary at
$0$ plays the role of a \emph{cut-off}, preventing the emergence of multiple
mergers in the genealogy (or equivalently,
the appearance of jumps in the limiting fluctuations of the system).

The proofs of our results rely on a method of moments for marked metric measure spaces,
developed in~\cite{foutel22}, to establish the joint convergence of
the population size and genealogy of a class of branching systems.
This method was previously applied
in~\cite{tourniaire2023tree, foutel2024convergence} to study the genealogy,
demography, and spatial configuration of a BBM with inhomogeneous branching rate,
negative drift, and killing at $0$. In that framework, the BBM served as a toy model
for the dynamics at the tip of an invasion front in a cooperating population.
For a certain parameter regime, known as the semipushed regime~\cite{tourniaire2021branching},
the reproductive variance $\Sigma^2$  (whose definition is as in \eqref{eq:def_Sigma})
scales like $N^c$ for some $c \in (0,1)$, as in
Proposition~\ref{prop:concentration}. In both the semipushed~\cite{tourniaire2021branching} and the
moderate mutation-selection regimes, the relevant timescale for
the scaling limit of demographic fluctuations and the genealogy is given by
$N/\Sigma^2 = N^{1-c}$.
However, the limiting behavior of the BBM
in~\cite{tourniaire2021branching, foutel2024convergence} differs significantly
from what we observe here. In semipushed waves, the demographic fluctuations are
described by a $(2-c)$-stable CSBP, and the genealogy converges to a
Beta$(c, 2 - c)$-coalescent.
The key distinction lies in the concentration of reproductive variance:
in~\cite{tourniaire2021branching,foutel2024convergence}, the mass
in the integral \eqref{eq:def_Sigma} is concentrated
in an interval $I$ where the expected number of individuals is microscopic,
meaning $N\int_I \tilde h < 1$, whereas in our setting, it is mesoscopic, i.e.,
$1 \ll N\int_I \tilde h \ll N$ (see~Proposition~\ref{prop:concentration2}).

Here, the limiting genealogy belongs to the universality class of neutral
evolution models with an effective population size. Typically (see \cite{powell19,tourniaire2023tree,forien2025stochastic}), such behaviour arises when the reproductive variance,
computed from the Perron–Frobenius eigenvectors of the branching mechanism,
is of the same order as the demographic parameter $N$.
Our analysis shows that this condition can be relaxed: if the reproductive variance, rescaled by the effective population size (which grows much more slowly than the actual population size), converges, then the same behaviour emerges

\section{Outline of the proof}

\subsection{The method of moments for metric measure spaces}
\label{sec:mm_space}

{As
	in~\cite{tourniaire2023tree,foutel2024convergence}, we enrich the structure of the branching diffusion to prove the joint convergence of the demographic fluctuations and the genealogy of the process. The BBM and its genealogy are encoded as a random metric measure space, and we prove convergence in the Gromov-weak topology using the method of moments introduced in~\cite{foutel22}.}

A metric measure space $M=(X, d, \mu)$ is  a complete separable metric
space $(X,d)$ equipped with a finite measure~$\mu$. Let $|X|:=\mu(X)$ and note
that we do not require $\mu$ to be a probability measure. We say that $(X,d,\mu)$
and $(X',d',\mu')$ are equivalent if there exists an isometry $\varphi$ between
the supports of $\mu$ and $\mu'$ such that $\mu'$ is the pushed forward of $\mu$
by $\varphi$. We denote by $\mathbb{M}$ the set of equivalence
classes of mm-spaces.
\begin{definition}[Polynomials \textnormal{\cite[Definition 2.3]{greven2009convergence}}] \label{def:polynomials}
	A functional $\Phi:\mathbb{M}\to\mathbb{R}$ is a polynomial of degree $k\in\mathbb{N}$
	if there exists a bounded continuous function
	$\phi:[0,\infty)^{\binom{k}{2}}\to \mathbb{R}$ such that
	\begin{equation} \label{def:polynomial}
		\Phi((X,d,\mu))=\int_{X^k} \phi(d(v_i,v_j)_{1\leq i<j\leq k}))
		\prod_{i=1}^{k}\mu(dv_i).
	\end{equation}
	We write $\Pi$ for the set of all polynomials.
\end{definition}

\begin{definition}[Gromov-weak topology \textnormal{\cite[Definition 2.8]{greven2009convergence}}] The Gromov-weak topology is defined as the topology induced by the polynomials, that is,	the smallest topology making all polynomials continuous. A sequence $(X^n,d^n,\mu^n)$ is said to converge  to $(X,d,\mu)$ in $\mathbb{M}$ with respect to the Gromov-weak topology if and only if $\Phi(X^n,d^n,\mu^n)$ converges to $\Phi(X,d,\mu)$ for all $\Phi\in \Pi$.
\end{definition}

\begin{definition}\label{def:gromov-topology}
	A random mm-space is a random variable with values in $\mathbb{M}$, endowed
	with the Gromov-weak topology and the associated Borel $\sigma$-field.
\end{definition}

Many properties of the marked Gromov-weak topology are derived in
\cite{depperschmidt_marked_2011} under the additional assumption that
$\mu$ is a probability measure. The next result shows that ${\Pi}$ forms
a convergence determining
class only when the limit satisfies a moment condition, which is a
well-known criterion for a real variable to be identified by its moments,
see for instance \cite[Theorem~3.3.25]{durrett_probability_2019}.

\begin{definition}
	For $\Phi\in\Pi$ and a random mm-space $(X,d,\mu)$, the quantity
	$\mathbb{E}[(\Phi(X,d,\mu))]$ is called the moment of $(X,d,\mu)$ associated to $\Phi$.
\end{definition}

\begin{proposition}[Method of moments
		\textnormal{\cite[Lemma~2.7]{depperschmidt2019treevalued}}] 	\label{lem:convDetermining}

	Suppose that $(X,d,\mu)$ is a random mm-space satisfying
	\begin{equation} \label{eq:momentCondition}
		\limsup_{p \to \infty} \frac{\E[|X|^p]^{1/p}}{p} < \infty.
	\end{equation}
	Let $((X^n, d^n, \mu^n))$ be a sequence of random mm-spaces. If
	\[
		\lim_{n \to \infty} \E\big[ \Phi\big(X^n, d^n, \mu^n\big) \big]
		= \E\big[ \Phi\big(X, d, \nu\big) \big],
	\]
	for all $\Phi \in {\Pi}$, then, the sequence $((X^n, d^n, \mu^n]))$
	converges in distribution to
	$(X,d,\nu)$ for the marked Gromov-weak topology.
\end{proposition}

\subsection{The Metric Space associated to the BBM}\label{sec:mm}
Let $t>0$ and recall that ${\mathcal N}_{t}$ denotes the set of particles alive at time $t$ in the BBM $\mathbf{X}^N$ (recall that we drop the $N$-superscripts for simplicity). Define
\begin{equation}
	\mu_{t} \ :=  \sum_{v\in {\mathcal N}_{t}} \delta_{v}, \quad \text{and}
	\quad
	\forall v,v'\in{\mathcal N}_{t}, \ \
	d_t(v,v') = \ t - |v\wedge v'|, \label{eq:distance_mesure}
\end{equation}
where $|v\wedge v'|$ denotes the most recent time when $v$ and $v'$ had a common
ancestor. Let $M^N_{t}: = [{\mathcal N}_{t},d_t ,\mu_{t}]$ be the resulting random mm-space.
Finally, set
\begin{equation*}
	\bar \mu_{t} \ := \frac{1}{N} \sum_{v\in {\mathcal N}_{tN^{1-c}}} \delta_{v},
	\quad \text{and} \quad  \forall v,v'\in{\mathcal N}_{tN^{1-c}}, \quad
	\bar d_t(v,v') =  \left(t - \frac{1}{N^{1-c}}|v\wedge v'|\right),
\end{equation*}
and define the rescaled mm-space
$\bar M_t  :=  ({\mathcal N}_{tN^{1-c}},\bar d_t ,\bar \mu_{t})$.
In order to prove Theorem \ref{th:genealogy}, we will show that the rescaled mm-space
$\bar M_t$ converges to a limiting mm-space defined in the next section.

\subsection{The Brownian Coalescent Point Processes} \label{sect:CPP}
In this section, we introduce the Brownian Coalescent Point Processes (CPP).
As we  shall see, this object is the scaling limit of the mm-space associated to the BBM $\mathbf{X}^N$.

Let $T>0$.
Consider ${\cal P}$ a $PPP\left(\frac{dt}{t^2} \otimes d\ell\right)$.
Define
$$
	V_T \ = \ \inf\left\{y : (t,y)\in {\cal P}, t\geq T\right\},
$$
and
$$
	d_T(x,y) \ = \ \sup\{t  :  (t,z) \in {\cal P}\ \ \mbox{and} \ \ x \leq z \leq y\}, \quad  0<x<y<V_T.
$$
The Brownian Coalescent Point Process~\cite{popovic2004} is defined as
$$
	M_{\text{CPP}_T} \ :=  \bigg([0,V_T],d_T, \text{Leb}  \bigg).
$$

\begin{rem}\label{rem:exp}
	A direct computation shows that $V_T$ (which can be thought of as the population
	size at time $T$) is distributed as an exponential random variable
	with mean $T$. If the CPP encodes the size and the genealogy of critical
	branching processes, this is consistent with Yaglom's law for such processes.
\end{rem}

\begin{proposition}\label{moments:CPP}
	Let $k\in{\mathbb N}$.
	For any bounded continuous function $\phi:[0,\infty)^{\binom{k}{2}}\to \mathbb{R}$ with associated polynomial $\Phi$,
	we have
	$$
		\E\left[\Phi(M_{\text{\textnormal{CPP}}_T}) \right] \ = \ k ! T^{k} \E\left[\Phi( U_{\sigma(i),\sigma(j) } )   \right] ,
	$$
	where
	$(U_{i}; i\in[k-1])$ is a vector of uniform i.i.d. random variables on $[0,T]$,
	$$U_{i,j}= U_{j,i} = \max\{ U_{l}: l=i,\cdots, j-1\},$$ $\sigma$ is a random uniform permutation of $\{1, \dots, k\}$, and $\sigma$ and $(U_i)$ are independent.
\end{proposition}
\begin{proof}
	The proof of this result can be found in \cite{foutel22}, Proposition 4.
\end{proof}

{
\begin{proposition}[Sampling from the CPP]\label{SAmpling-CPP}
	Let $k\in\N$ and
	sample $k$ points, denoted by $v_1,\cdots,v_k$, uniformly at random from the CPP.
	Then $(d_T(v_i,v_j))$ is identical in law to  $(H_{\sigma(i),\sigma(j)})$
	where $(H_{i,j})$ be as in \eqref{eq:moments-CPP} (ii)
	and $\sigma$ is a random uniform permutation of $\{1,\cdots,k\}$, independent of $(H_{i,j})_{i,j}$.
\end{proposition}
\begin{proof}
	The proof is identical to~\cite[Proposition 4.3]{boenkost2022genealogy}.
\end{proof}
}
\begin{theorem}\label{th:main-theorem}
	Let $t>0$ and let $\sigma$ be as defined in Proposition~\ref{prop:concentration}. Let $M_{CPP_T}$ be a Brownian CPP with height $T=\frac{\sigma^2}{2}t$.
	Conditional on the event $\{Z_{tN^{1-c}}>0 \}$, the sequence
	$(\bar M_{t}; N\in\N)$ converges weakly to $M_{CPP_T}$
	with respect to the Gromov-weak topology.
\end{theorem}

\subsection{Marked metric measure spaces and the many-to-few formula}\label{sec:spine}

As explained above, our proof strategy consists in computing the moments of the
mm-space $\bar{M}_t$ associated with the BBM. To this end, we rely on the so-called
many-to-few formula, which expresses these moments in terms of the $k$-spine of the BBM.
Before introducing the $k$-spine tree and stating the many-to-few formula, we
must first enrich the structure of the mm-space by keeping track of particle marks.
For this purpose, we work with marked metric measure spaces, which we now introduce.

Let $(E,d_E)$ be a complete metric space (the mark space). A mmm-space is a triple
\[
	\mathscr{M}=(X,d,\vartheta),
\]
where $(X,d)$ is a complete separable metric space and $\vartheta$ is a finite measure on $X\times E$.
As in the unmarked setting, one may define a topology on the set (of equivalence classes)
of mmm-spaces $\mathbf{M}$ through polynomial functionals
\[
	{\bf \Phi}(\mathscr{M})
	= \int_{(X\times E)^k}
	{\bf \phi}\!\left( (d(v_i,v_j); x_i)_{1\le i<j\le k} \right),
\]
where $\phi:[0,\infty)^{\binom{k}{2}}\times E^k \to \mathbb{R}$ is any continuous bounded test function.
	{As in the unmarked case, the marked Gromov-weak topology is induced by polynomials.
		Moreover, a random mmm-space is a $\mathbf{M}$-valued random variable, where
		$\mathbf{M}$ is endowed with the marked Gromov-weak topology and its associated
		Borel $\sigma$-field. The moment of the random mmm-space $\mathscr{M}$ associated
		with the polynomial ${\bf \Phi}$ is then defined as $\mathbb{E}[{\bf \Phi}(\mathscr{M})]$.
	}

For $t\ge 0$, define the empirical measure and its rescaled version by
\[
	\vartheta_t := \sum_{v\in{\mathcal{N}}_t} \delta_{v,x_v},
	\qquad
	\bar{\vartheta}_t := \frac{1}{N}\sum_{v\in \mathcal{N}_{tN^{1-c}}} \delta_{v,x_v}.
\]
The spaces
\[
	\mathscr{M}_t := (\mathcal{N}_t, d_t, \vartheta_t),
	\qquad \text{and} \qquad
	\bar{\mathscr{M}}_t := (\mathcal{N}_{tN^{1-c}}, \bar d_t, \bar \vartheta_t),
\]
are marked metric measure spaces.

Next, we introduce the $k$-spine tree associated to the BBM and state
our many-to-few formula.
\begin{definition}\label{def:1-spine}
	The spine process $(\zeta_t)_{t\geq 0}$ is the diffusion with generator
	\begin{equation}
		\frac{1}{2} f''(x) \ + \frac{v_1'(x)}{v_1(x)} f'(x), \qquad \ f'(0)=0, \quad f(\lcn)=0. \label{eq:gen-spine}
	\end{equation}
	Let $q_t(x,y)$ denote the probability kernel of the spine process.
	The generator \eqref{eq:gen-spine} is the Doob $h$-transform of
	the differential operator $\frac{1}{2}\partial_{xx}+\beta\partial_x+\frac{1}{2}$. In particular,
	\begin{equation}
		\label{eq:rel_qp}
		q_t(x,y)=\frac{h(y)}{h(x)}p_t(x,y),
	\end{equation}
	where $p$ is as in Lemma~\ref{lem:many-to-one0}.
\end{definition}

The next result is standard (see e.g.~\cite{etheridge_2011}).
\begin{proposition}\label{prop:invariant-1-spine}
	The spine process has a unique invariant measure given by
	\begin{equation}\label{eq:def_pi}
		\Pi(dx)  =  h(x)\tilde{h}(x) dx.
	\end{equation}
\end{proposition}
We now move to the definition of the $k$-spine tree.  Let $(U_1,...,U_{k-1})$ be independent random variables uniformly distributed in $[0,t]$. Define
\begin{equation}
	\label{eq:UPP}
	\forall \; 1\leq i < j\leq k-1,\quad U_{i,j}=U_{j,i}= \;\max\{U_i,...,U_{j-1}\}.
\end{equation}
Let ${\mathbb T}$ be the unique tree of depth $t$ with $k$ leaves such that the tree distance between the $i$-th and the $j$-th leaves is given $U_{i,j}$. This tree is \emph{ultrametric} and \emph{planar} in the sense that
\begin{equation*}
	\forall i,j,l\in[k], \quad U_{i,j}\leq U_{i,l}\vee U_{l,j},
\end{equation*}
(ultrametric) and the inequality becomes an equality if $i<l<j$ (planar).
The depth of the first branching point in the $k$-spine tree is thus given by
\begin{equation}
	\label{eq:def_tau}
	\tau=\max_{i\in[k-1]}U_i.
\end{equation}
Marks are then assigned as follows. On each branch of the tree, the mark evolves
according to the spine process $\zeta$ and it branches into two independent
diffusions at each branching point of $\mathbb{T}$. The resulting planar marked
ultrametric tree will be referred to as the
$k$-spine tree and denoted by $\mathcal{T}$.
We write $\mathcal{T}_0$ (resp.~$\mathcal{T}_1$) for the left (resp.~right) marked
subtree attached to the first branching point of the tree $\mathcal{T}$.
Note that these two subtrees are also planar marked
ultrametric trees.
We will denote by $Q^{k,t}_x$ the distribution of the  $k$-spine tree
of height $t$ rooted at $x$.
\begin{figure}
	\centering
	\includegraphics[width=.5\textwidth]{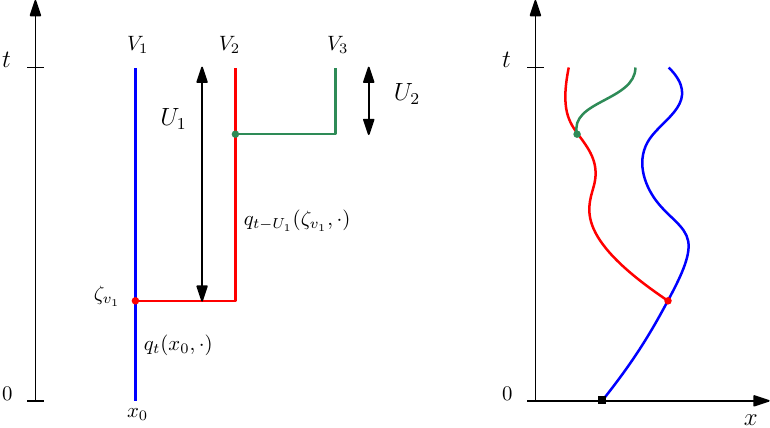}
	\caption{$k$-spine tree with $k=3$. Left panel: planar tree $\mathbb{T}$ generated
		from $2$ i.i.d. uniform random variables $(U_1,U_2)$.
		Right panel: branching $1$-spines running along the branches of the tree
		${\mathbb T}$.
	}
	\label{fig:k-spine}
\end{figure}
In the following, $\mathcal B$ will denote the set of  $k-1$ branching points of
the $k$-spine and $\mathcal L$ will denote the set of $k$ leaves.
We will denote by $\zeta_v$ the mark (or the position) of the node
$v\in \mathcal B\cup \mathcal L$.
Finally, $(V_i; i\in[k])$ is the enumeration of the leaves from left to right
in the $k$-spine (i.e., $V_i$ is the leaf with label $i$). See Figure~\ref{fig:k-spine} for an illustration of these definitions. We refer to~\cite{foutel22}
for a formal construction of the $k$-spine tree.

\begin{theorem}[Many-to-few formula]
	\label{many-to-few00}
	{Let $k\in\mathbb{N}$, $t>0$ and $x\in(0,L)$.
	Let $\phi:[0,\infty)^{\binom{k}{2}}\to \mathbb{R}$ and $f_i:E\to \R$, $i\in[k]$, be bounded
	measurable functions and define
	\[
		\forall \mathscr{M}=(X,d,\vartheta)\in \mathbf{M}, \quad \tilde{\bf \Phi}(\mathscr{M})
		=\int_{ (X\times E)^k} \; \mathbf{1}_{\{(v_i) \text{ distinct }\}}
		\phi(d(v_i,v_j)_{1\leq i<j\leq k})
		\prod_{i=1}^{k}f_i(x_i)\vartheta(dv_i\otimes dx_i).
	\]
	Then we have}
	\begin{equation*}
		\mathbb {E}_x\left[ \tilde{\bf \Phi}(\mathscr{M}_t) \right]  =  k !
		h(x) t^{k-1} Q_{x}^{k,t}\left( \Delta
		\phi(( U_{\sigma(i),\sigma(j)}))\prod_{i=1}^k f_i\left({\zeta}_{V_{\sigma(i)}}\right)\right)
		\quad \text{with} \quad \Delta
		:= \left(\frac{1}{2}\right)^{k-1}\prod_{v\in {\mathcal B}}  h( {\zeta}_{v})
		\prod_{i=1}^k \frac{1}{h( {\zeta}_{V_i})},
	\end{equation*}
	where $(U_{i,j})$ is as in \eqref{eq:UPP} and
	$\sigma$ is an independent random permutation of $[k]$.
\end{theorem}

{The proof of this is the object of Section~\ref{sec:proof_many_to_f}.}
\begin{definition}[Accelerated $k$-spine]\label{def:acc:spine}
	Consider the spine process accelerated by $N^{1-c}$, i.e. the transition
	kernel of the spine process
	is now given by $q^N_{tN^{1-c}}(x,y)$. {We denote this kernel by
			$\bar{q}^N_t(x,y)$}.
	Consider the same planar structure as before, i.e., the depth is $t$
	and the distance between points at time $t$ is given by (\ref{eq:UPP}).
	We denote by $\bar Q_{x}^{N,k,t}$ the distribution of the  $k$-spine tree
	obtained by running accelerated  spines along the branches.
	For any vertex $v$ in the accelerated $k$-spine tree, $\bar \zeta_v$ will denote the mark of the vertex $v$.
	Finally, we consider the family of measures $(\hat{Q}^{k,t}_x)$ defined  by
	\[
		\frac{d\hat{Q}^{k,t}_x}{d\bar{Q}^{k,t}_x}=
		\left(\frac{t}{N^{c}}\right)^{k-1}.
	\]
\end{definition}

\begin{proposition}[Rescaled many-to-few formula]\label{th:many-to-few2}
	{Let $k\in\mathbb{N}$, $t>0$ and $x\in(0,L)$.
	Let $\phi:[0,\infty)^{\binom{k}{2}}\to \mathbb{R}$ and $f_i:E\to \R$, $i\in[k]$, be bounded
	measurable functions and define
	\[
		\forall \mathscr{M}=(X,d,\vartheta)\in \mathbf{M}, \quad \tilde{\bf \Phi}(\mathscr{M})
		=\int_{ (X\times E)^k} \; \mathbf{1}_{\{(v_i) \text{ distinct }\}}
		\phi(d(v_i,v_j)_{1\leq i<j\leq k})
		\prod_{i=1}^{k}f_i(x_i)\vartheta(dv_i\otimes dx_i).
	\]
	Then we have}
	\begin{equation*}
		\mathbb {E}_x\left[\tilde  {\bf\Phi}(\bar{\mathscr{M}}_t) \right] =
		\frac{1}{N} k !
		h(x)  \hat{Q}_{x}^{k,t}\left( \bar \Delta
		\phi(( U_{\sigma(i),\sigma(j)}))\prod_{i=1}^k f_i\left({\zeta}_{V_{\sigma(i)}}\right) \right)
		\quad \text{with} \quad \bar\Delta
		:= \left(\frac{1}{2}\right)^{k-1}\prod_{v\in {\mathcal B}}  h( {\bar \zeta}_{v})
		\prod_{i=1}^k \frac{1}{h( \bar{\zeta}_{V_i})},
	\end{equation*}
	$(U_{i,j})$ is as in \eqref{eq:UPP} and
	$\sigma$ is an independent random permutation of $\{1,...,k\}$.
\end{proposition}

\begin{proof}
	This is a direct consequence Theorem \ref{many-to-few00} after
	rescaling the measure $\vartheta_t$ by $N$ and time by $N^{1-c}$.
\end{proof}

The proof of Theorem \ref{th:main-theorem} relies on the following convergence
result for the rescale $k$-spine measure $\hat Q_x^{k,t}$.
\begin{proposition}
	\label{th:k-spine-cv}
	Let $t>0$ and $x>0$. Let $k\in\mathbb{N}$.
	For any bounded continuous function $\phi:[0,\infty)^{\binom{k}{2}}\to \mathbb{R}$, we have
	\begin{equation*}
		\hat Q_{x}^{k,t}\left(  \bar{\Delta}
		\phi(( U_{\sigma(i),\sigma(j)}))  \right)
		\xrightarrow{N\to\infty} \left(\frac{\sigma^2t}{2}\right)^{k-1} {\mathbb E}
		\left( \phi ((U_{\sigma(i),\sigma(j)})) \right),
	\end{equation*}
	where $\sigma$ is a uniform permutation of $\{1,\dots,k\}$ independent of $\{ U_{i,j}; i,j \leq k \}$.
\end{proposition}

{We conclude this section with two remarks. First, note that the method of moments
in Proposition~\ref{lem:convDetermining} is stated for mm-spaces only, and not for
mmm-spaces. Second, observe that any moment of the mm-space $\bar{M}_t$ in
Proposition~\ref{th:k-spine-cv} can be computed by setting all the functions $f_i$
equal to~$1$ (as done in Proposition~\ref{th:k-spine-cv}).
}

\subsection{Outline of the proof and heuristics}\label{sec:heuristics}
Theorem~\ref{th:Yaglom} and Theorem~\ref{th:genealogy} can be deduced from
Theorem~\ref{th:main-theorem}. The idea behind the proof of
Theorem~\ref{th:main-theorem} consists in identifying the limiting moments of
the mm-space $\bar{M}_t^N$ to that of a Brownian CPP of height {$T=\frac{\sigma^2}{2}t$}
(see Theorem~\ref{th:main-theorem}) using a spinal decomposition
(see Proposition~\ref{th:many-to-few2} and Proposition~\ref{th:k-spine-cv}).
The convergence criterion for mm-spaces given in
Proposition~\ref{lem:convDetermining} then yields the result.

\paragraph{Heuristics on lineage dynamics via moment calculations}
Before turning to the strategy for proving Proposition~\ref{th:k-spine-cv},
we first discuss what the moment calculations indicate about the dynamics of
the population’s ancestral lineages.
First, recall that $h(x)$ denotes the expected number of descendants of a particle initially located at $x$ after a large time, see \eqref{eq:reproductive_value}.		One can check that, for $x \in [0,A]$,
\begin{equation}
	h(x) \approx N^c,
	\label{eq:heuristics_2}
\end{equation}
and that for all $1\ll z \ll L$, $h(z)\ll h(0)$.

Next, we examine the second moment of the population size.
By \cref{th:many-to-few2}, for $k=2$, $\phi\equiv 1$, $f_i\equiv 1$, and $t>0$,
we have
\begin{align}
	\E_x\!\left[{\bf\Phi}(\bar{\mathscr{M}}_t) \right]
	=\E_x\left[\frac{1}{N^2}\sum_{\substack{v_1\neq v_2                     \\\in \mathcal{N}_{tN^{1-c}}}}1\right]
	 & = \frac{2}{N}h(x)\,\,\hat{Q}_x^{2,t}(\Delta)
	=\frac{2}{N}h(x)\,\frac{t}{N^{c}}\,\bar{Q}_x^{2,t}(\Delta),   \nonumber \\
	 & = \frac{2}{N}h(x)\,\frac{1}{N^{c}} \Bigg[\int_{0}^{t} \int_{0}^{L}
		\bar{q}_s(x,y)h(y) \Big(\int_0^L \frac{1}{h(z)}\bar{q}_{t-s}(y,z)\,dz \Big)^2 dy\, ds \Bigg].
	\label{eq:heuristics_1}
\end{align}
Intuitively, the many-to-two formula
allows one to compute second moments in branching processes by following two distinguished lineages.
These lineages coincide until they split at their most recent common ancestor (MRCA), after which they evolve independently.
The term in square brackets in \eqref{eq:heuristics_1} can be interpreted as follows:
we first integrate over the age $s$ of the MRCA, and then over its position $y$.
The last integration corresponds to the trajectories of the two lineages
after they split, between times $s$ and $t$. (Note that this can be generalised
to $k$-th moments with the many-to-few formula, see Theorem~\ref{many-to-few00}.)
{As we shall see in Section~\ref{sec:hk}, the relaxation time of the spine process is of order $\log(N)^2$, in the sense that
		\begin{equation}
			\label{eq:mixing_time}
			\quad q_t(x,y)\approx \Pi(y), \quad \forall t \gg \log(N)^2.
		\end{equation}
		Since in \eqref{eq:heuristics_1}, the spine kernel is accelerated by the factor $N^{1-c}$, this suggests that, for $k=2$, $\phi\equiv 1$ and $t>0$,
		\begin{align}
			\E_x\!\left[{\bf\Phi}(\bar{\mathscr{M}}_t) \right]
			 & \approx \frac{2t}{N}h(x)\cdot\overbrace{\Big(\frac{1}{N^{c}}\int_{0}^L h(y)^2 \tilde h(y)\,dy\Big)}^{\mbox{$\approx \sigma^2$}}\,
			\cdot\underbrace{\Big(\int_0^L \tilde h(z)\,dz \Big)^2}_{\mbox{=1}},
			\label{eq:heuristics_3}
		\end{align}
		where we used the Perron--Frobenius normalisation \eqref{eq:normalisation}, Proposition~
		\ref{prop:concentration}, and the identity $\Pi=h\tilde h$.}

	{As explained above, the second factor in \eqref{eq:heuristics_3} is related to the typical position of the MRCA of two individuals.
		Yet, we know from Proposition~\ref{prop:concentration} that the mass of this integral is concentrated on the segment $[0,A]$, suggesting that the MRCA of two randomly sampled individuals lies in the best class.}
This, combined with \eqref{eq:heuristics_2}, suggests the following picture for the backward dynamics.
Sample $k$ individuals from the population at time $tN^{1-c}$.
Looking backward in time, the ancestral lineages of these individuals trace back to the best class $[0,A]$ within a time of order $\log(N)^2$.
Within this class, all individuals have the same reproductive value (see \eqref{eq:heuristics_2}), so the lineages evolve as in a neutral population of size given by the size of the best class, namely $N^{1-c}$, see Proposition~\ref{prop:concentration}.
This is illustrated in Figure~\ref{fig:fitness_class_coal}.

Intuitively, we expect our model to behave similarly to the following unstructured
(i.e., non-spatial)
model. Consider a Cannings model with population size $N$, where the offspring
vector $\nu^N = (\nu^N_1,\dots,\nu^N_N)$ in each generation is a random
permutation of the vector
\begin{align*}
	v=( \underbrace{N^c,\dots,N^c}_{N^{1-c} \text{ coordinates}},0,\dots,0).
\end{align*}
For simplicity, we assume that all powers are natural
numbers. Möhle's Lemma~\cite{Moehle1998} states that if the ratio between
the triplet coalescence probability $d_N$ and the pair coalescence probability
$c_N$ vanishes, then the genealogy converges to a Kingman coalescent on the
time scale $c_N^{-1}$. Here, we have
\begin{align*}
	\frac{d_N}{c_N}
	=
	\frac{\EE{\sum_{i=1}^{N}
			\frac{\nu^N_i (\nu^N_i-1) (\nu^N_i -2)}{N (N-1)(N-2)} }}
	{\EE{\sum_{i=1}^{N}
			\frac{\nu^N_i (\nu^N_i-1)}{N (N-1)} }}
	\sim
	\frac{\EE{(\nu^N_1)^3}}{N \EE{(\nu^N_1)^2}}
	\sim N^{c-1}.
\end{align*}
Moreover, $c_N \sim N^{c-1}$, suggesting that Möhle's condition is satisfied
and that the genealogy is given by a Kingman coalescent on the time scale
$c_N^{-1} \sim N^{1-c}$.

\begin{figure}[!b]
	\centering
	\begin{minipage}{0.45 \textwidth}
		\centering

		\includegraphics{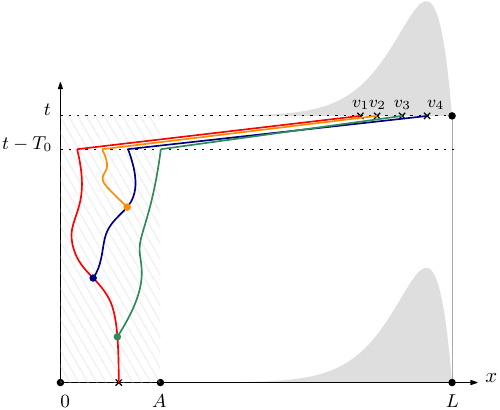}

	\end{minipage}
	\begin{minipage}{0.45 \textwidth}
		\centering

		\includegraphics[width=1\textwidth,trim = 0 99 0 50,clip]{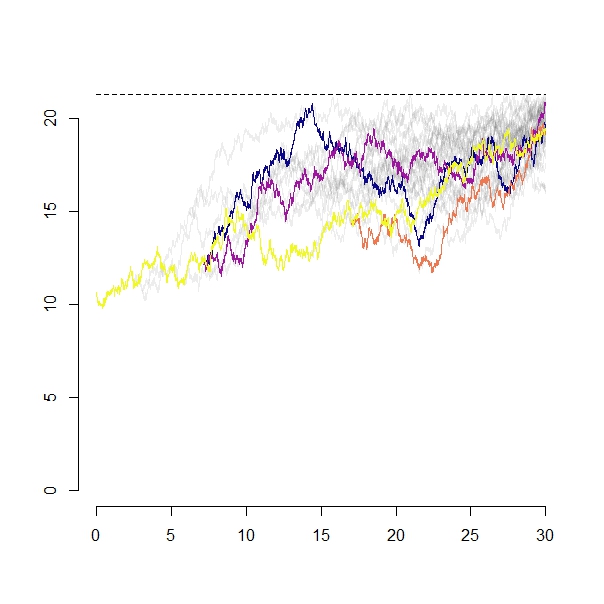}

	\end{minipage}
	\caption{Left panel: Spatial trajectories of the ancestral lineages of $4$ individuals sampled
		at time $tN^{1-c}$.
		All lineages trace back to the best class $[0,A]$ within $T_0 = O(\log(N)^2)$.
		Right panel: Simulation of the BBM for $\beta = 0.99$.
		The dashed line indicates the absorbing boundary at $L \approx 21.27$.
		Coloured trajectories correspond to the ancestral lineages of $4$ particles sampled
		uniformly at the end of the simulation ($t=20$), while the grey trajectories in
		the background show $5\%$ of the surviving lineages.
	}
	\label{fig:fitness_class_coal}
\end{figure}

\paragraph{Convergence of moments}
We now give  brief heuristics to explain why Proposition~\ref{th:k-spine-cv}
should hold. By definition of ${\Delta}$,
\begin{equation*}
	\hat Q_{x}^{k,t}\left(  {\bar\Delta }\cdot \phi( (U_{\sigma(i),\sigma(j)}))    \right)=
	\hat Q_{x}^{k,t}\left( \left(  \prod_{v\in \mathcal B}{h(\bar {\zeta}^N_{v})}
	\prod_{ i=1}^k \frac{1}{h(\bar {\zeta}_{V_{\sigma(i)}})} \right)
	\cdot \phi(( U_{\sigma(i),\sigma(j)}) )
	\right).
\end{equation*}
By definition of the accelerated $k$-spine tree, its structure is binary a.s.~and
the marks along the branches are given by spine processes accelerated by $N^{1-c}$.
On the other hand,  we know from \eqref{eq:mixing_time} that the mixing time
of the spine process is much smaller than $N^{1-c}$.
As a consequence, the variables $(\bar \zeta_u)_{u\in\mathcal{B}\cup\mathcal{L}}$ should be
well-approximated by $2k-1$ i.i.d.~random variables with law $\Pi$.
Hence, by the same arguments used to obtain \eqref{eq:heuristics_3}, we see that
\begin{align}
	\hat Q_{x}^{k,t}\left(   \Delta \cdot \phi( (U_{\sigma(i),\sigma(j)}))    \right)
	 & \approx \frac{t^{k-1}}{N^{(k-1)c}}	\left(\frac{1}{2}{\int_0^\lcn (h(y))^2\tilde{h}(y)dy}\right)^{k-1}
	\left( \int_0^{\lcn}\tilde{h}(y)dy\right)^k{\mathbb E}\left[ \phi( (U_{\sigma(i),\sigma(j)}) )\right]
	\nonumber ,                                                                                                   \\
	 & \approx\left(\frac{\sigma^2}{2} t\right)^{k-1} {\mathbb E}\left[ \phi(( U_{\sigma(i),\sigma(j)})) \right].
	\label{eq:mixing}
\end{align}
This yields the content of
Proposition~\ref{th:k-spine-cv}.
Once this result is proved, Theorem~\ref{th:many-to-few2} then shows that,
for all polynomials $\Phi$, as $N\to\infty$,
\begin{equation*}
	\E_x[\Phi(\bar{M}_t)]\approx
	\frac{1}{N} k! \left(\frac{\sigma^2}{2} t\right)^{k-1}h(x)
	{\mathbb E}\left[ \phi( (U_{\sigma(i),\sigma(j)})) \right].
\end{equation*}
We then see from Theorem~\ref{th:Kolmogorov} that
\begin{equation*}
	\E_x[\Phi(\bar{M}_t)| Z_{tN^{1-c}}>0]\approx k!
	\left(\frac{\sigma^2}{2} t\right)^{k}
	{\mathbb E}\left[ \phi(( U_{\sigma(i),\sigma(j)}) )\right].
\end{equation*}
Proposition~\ref{lem:convDetermining} then implies that $\bar{M}_t$ converges to a
Brownian CPP of depth $\frac{\sigma^2}{2} t$.

\paragraph{The Kolmogorov estimate}

Our proof of the Kolmogorov estimate (Theorem~\ref{th:Kolmogorov}) also requires sharp
bounds on the moments of the population size, obtained via a many-to-few formula.
However, the many-to-one formula used here is not the one stated in Theorem~\ref{many-to-few00}. Instead, it relies on a different construction of the spine process, going back
to the original work of Lyons, Pemantle and Peres~\cite{lyons1995conceptual},
and later developed by Harris and Roberts~\cite{harris2017spine}.

Although this change of strategy may appear to reveal a limitation of the method of moments, we argue that both approaches are natural and complementary. Indeed, it is not surprising
that the Kolmogorov estimate requires finer control, as this result is strongly model-dependent.
By contrast, the convergence of the unconditioned moments is more robust and relies
only on two spectral properties of the process: fast mixing guarantied by the control on the spectral gap, and suitable estimates on the Green's function.

Moreover, the method of moments
directly captures the limiting genealogical structure of the population and
points toward a natural sampling procedure, which is not the case for the spine
construction used in the proof of the Kolmogorov estimate.

\paragraph{General method for branching diffusion with an effective population size}
We conjecture that the method outlined above can be extended to a broad class of
spatial branching processes, which we will refer to as branching diffusions with an effective population size. Essentially, we expect the method to apply to any branching diffusion
satisfying the following conditions:
(1) The principal eigenvalue of the generator of the branching diffusion is zero, and the associated left and right eigenfunctions $h$ and $\tilde{h}$ can be normalized as in \eqref{eq:normalisation}.
Moreover, the reproductive variance $\Sigma^2$ (see \eqref{eq:def_Sigma}) scales as $N^c$ for some
$c \in (0,1)$.
(2) The moments of the $k$-spine tree associated to the diffusion converge to the
moments of a Brownian CPP after rescaling time by $N^{1-c}$ and the size of the
process by $N$.
(3)
The probability that a system started from a single particle at $x$ survives up
to time $t N^{1-c}$ is proportional to $\frac{h(x)}{N}$.

As we shall see in the proofs, all these conditions should arise from spectral
properties of the process. The second condition follows from the fast mixing of
the spine process combined with appropriate estimates on the Green’s function
(see Section~\ref{sec:cv:mom}). The third condition reflects the fact that the
first and second moments are dominated by the $O(N^{1-c})$ particles located near $0$ (see~Proposition~\ref{prop:concentration2} and~\cref{lem:uba} in \cref{sec:Kolmogorov}).
We leave the formalization of this conjecture for future work.
\paragraph{The case of fitness waves}
Finally, we  outline how the proof extends from the case of a BBM starting from a single particle to a front-like initial configuration (see Remark~\ref{rk:Feller}).
The convergence of the demographic fluctuations for the BBM started from $N$ particles follows from the Kolmogorov estimate (Theorem~\ref{th:Kolmogorov}) and the convergence of the moment associated to $\phi\equiv 1$, provided that these convergence results are uniform in the starting point $x$ of the BBM. The above calculations (applied to $\phi\equiv 1$) indicate that, for the BBM started from a single particle at $x>0$, the process $\bar{Z}_t=\frac{1}{N}Z_{tN^{1-c}}$ is well-approximated by a random variable with law
\begin{equation*}
	\left( 1-\frac{2}{\sigma^2t}\frac{h(x)}{N}\right)\delta_0(dx)+\frac{2}{\sigma^2t}\frac{h(x)}{N}\exp\left(-\frac{2x}{\sigma^2t}\right)\frac{2dx}{\sigma^2t}.
\end{equation*}
Hence, for $q\geq 0$
\begin{equation*}
	\mathbb{E}_x\left[e^{-q\bar{Z}_t}\right]\approx  \left( 1-\frac{2}{\sigma^2t}\frac{h(x)}{N}\right) + \frac{2}{\sigma^2t}\frac{h(x)}{N}\frac{\sigma^2t}{\sigma^2t+2q}=  1-\frac{2}{\sigma^2t}\frac{h(x)}{N} \frac{\frac{\sigma^2 t}{2}q}{1+\frac{\sigma^2 t}{2}q}\approx \exp\left(-\frac{h(x)}{N}\frac{q}{1+\frac{\sigma^2t}{2}}\right),
\end{equation*}
for $N$ large.
If we now assume that the sequence of initial configurations $\mathcal{N}_0$ are such that $\frac{1}{N}\sum_{v\in\mathcal{N}_0} h(x_v(0))$ converges to some $z_0>0$ in probability, the branching property shows that
\begin{equation*}
	\mathbb{E}\left[e^{-q\bar{Z}_t}\right] \approx \E_x\left[\prod_{v\in\mathcal{N}_0}\exp\left(-\frac{h(x_v(0))}{N}\frac{q}{1+\frac{\sigma^2t}{2}}\right)\right]\approx \exp\left(-z_0\frac{q}{1+\frac{\sigma^2t}{2}}\right),
\end{equation*}
which coincides with the Laplace transform of a Feller diffusion.

\section{The many-to-few formula}
\label{sec:proof_many_to_f}

The proof of Theorem \ref{many-to-few00}
relies on the branching property. This
property allows us to derive a recursion formula that holds both for the moments
of the branching diffusion and the law of the $k$-spine. The proof of the many-to-few formula
for branching diffusions killed
at the boundary of a domain $\Omega\subset \R^d$
can be found in \cite[Section 3.1]{tourniaire2023tree}.
The proof of~Theorem \ref{many-to-few00} follows essentially the same lines; however,
we outline the main steps, as some intermediate results will be needed later.

The first step relies on
a uniform planarisation of the BBM.
At every time $t>0$, every particle is endowed with a mark $x_v$
(the position of the particle) and a Ulam-Harris label $p_v$,
where $p_v\in \cup_{n\in \mathbb{N}} \{0,1\}^n$. As before, ${x_v}$ denotes the position of the particle. The planarisation labels $p_v$
are assigned recursively as follows. We label the root with $\emptyset$ and
\begin{enumerate}
	\item At every branching point $v$, we distribute the
	      labels $(p_v,0)$ and $(p_v,1)$ uniformly among the two children: $(p_v,0)$ (resp.~$(p_v,1)$)
	      is said to the left (resp.~right) child of $v$.
	\item The label  $p_v$ does not vary between two branching points, i.e., $p_{v_1}=p_{v_2}$
	      if the trajectory connecting $v_1$ and $v_2$ does not encounter any branching points.
\end{enumerate}
Let ${\cal N}_t^{pl}$  be the set of particles at time $t$ in the planar branching diffusion. The genealogical distances and the marks of the planar branching diffusion will be encoded by marked binary planar ultrametric matrices.
We say that a matrix $(U_{i,j})_{1\leq i,j\leq k}$ is \textit{planar ultrametric} if
\begin{equation*}
	\forall i<j<l, \quad U_{i,l}=U_{i,j}\vee U_{j,l}.
\end{equation*}
Moreover, the matrix $(U_{i,j})_{1\leq i,j\leq k}$ is said to be \textit{binary} if
\begin{equation*}
	\forall i<j<l, \quad |\{U_{i,j},U_{i,l},U_{j,l}\}|\geq 2.
\end{equation*}
We denote by $\mathbb U_k$ the set of binary  planar ultrametric matrices of size $k$. Let $\mathbb{U}_k^*=\mathbb{U}_k\times E^k$ be  the set of marked binary planar distance matrices.
Note that the Ulam-Harris $(p_v)$ labelling induces an order on ${\cal N}_t^{pl}$. In particular, for  every $k$-uplet $v_1<v_2\cdots< v_k$ in ${\cal N}_t^{pl}$ and   $\vec{v}=(v_1,...,v_k)$, the marked distance matrix of the sample $\vec{v}$,
\begin{equation*}
	U(\vec{v}):=\left((d_t(v_i,v_j),(x_{v_i}))\right),
\end{equation*}
is an element of $\mathbb{U}_k^*$.

Our recursion formula on the moments of $M_t$ is obtained by dividing every ordered
$k$-uplet in ${\cal N}_t^{pl}$ into two subfamilies,  the descendants of the left
(resp.~right) child of the Most Recent Common Ancestor (MRCA) of the sample. This
is by achieved by partitioning  $[k]$  as follows.
For  ${U}=((U_{i,j}),(x_i))\in\mathbb{U}_k^*$,
define
\begin{equation}
	\tau(U) =\max_{i\neq j} U_{i,j}.
\end{equation}
In words, $\tau(U)$ is the time to the MRCA of the sample.
We say  the integers $i$ and $j$ are in the same block iff $U_{i,j}<\tau$. Since $U$ is a binary planar ultrametric matrix, there exists $n\leq k-1$ such that this partition can be written as $\{\{1,...,n\},\{n+1,...,k\}\}$.
We  denote by $T_0(U)$ and $T_1(U)$ the corresponding sub-matrices obtained from this two-block partition and write $|T_0(U)|$ and $|T_1(U)|$ for the sizes of the two blocks.
Note that $U_{i,j}$ is equal to $\tau(U)$ if $i$ and $j$ do not belong to the same block, and to $(T_0)_{i,j}$ (resp.~$(T_1)_{i-|T_0|,j-|T_0|}$) if they both belong to the first (resp.~second) block.

The next step consists in deriving a recursion formula for functionals $F:\mathbb{U}_k^*\to\mathbb{R}$ of the product form
\begin{equation}
	\label{eq:prod_form}
	F(U) \ = \ \mathbf{1}_{\left\{|T_0|=k-n,|T_1|=n\right\}}f(\tau(U)) \psi_0(T_0(U)) \psi_1(T_1(U)),
\end{equation}
where   $n\in\{1,...,k-1\}$ and $f:\mathbb{R^+}\to \mathbb{R}$, $\psi_1:\mathbb U_n^*\to \mathbb{R}$, $\psi_0:\mathbb U_{k-n}^*\to \mathbb{R}$  are bounded measurable functions.

\begin{proposition}[Recursion formula for the planar moments]\label{lem:23455}
	Let $k\in\mathbb{N}$, $t>0$ and $x>0$.  Let $R^{k,t}_x$  be the measure on $\mathbb{U}_k^*$ such that for every bounded measurable function $F:\mathbb{U}_k^*\to\mathbb{R}$,
	$$
		R^{k,t}_x(F) \ := \frac{1}{h(x)} \mathbb{E}_{x}\left(\sum_{\substack{v_1 < \cdots < v_k\\ v_i \in {\cal N}_t^{pl}}} F( U(\vec{v}))     \right),
	$$
	where $\vec{v}=(v_1,...,v_k)$. Then, for all functional $F$ of the product form \eqref{eq:prod_form},
	$$
		R^{k,t}_x(F) =  \frac{1}{2}\int_0^t f(t-s)\int_0^L h(y)q_{t-s}(x,y)
		R_y^{n,s}(\psi_0)R_y^{k-n,s}(\psi_1)dy ds.
	$$
\end{proposition}
\begin{proof}
	The proof of this result is contained in the proof of Proposition 7 in~\cite{tourniaire2023tree}
\end{proof}
The next step in the proof of Theorem~\ref{many-to-few00} is to show that
$(t^{k-1}Q^{k,t}_x)$ satisfies a similar recursive relation. We refer to \cite[Section 3.1]{tourniaire2023tree} for a proof of this fact.
\begin{proposition}[Recursion formula for the spine measure]
	\label{prop:rec_spine}

	Let $t>0$ and $x\in \Dcn$.
	Recall the definition of $\Delta$ from Theorem~\ref{many-to-few00}.
	For every functional $F$ of the product form \eqref{eq:prod_form},
	\begin{equation*}
		t^{k-1}Q_x^{k,t}(\Delta F)=\frac{1}{2}
		\int_0^t f(t-s)\int_0^L h(y)q_{t-s}(x,y)(s^{n-1}Q_y^{n,s}(\Delta \psi_0))
		(s^{k-n-1}Q_y^{k-n,s}(\Delta \psi_1))dy ds,
	\end{equation*}
	where, by a slight abuse of notation, we write
	\[Q_x^{k,t}(\Delta F)=Q_x^{k,t}(\Delta F((U_{i,j}),(\zeta_{V_i}))),\]
	and $(U_{i,j})$ is as in \eqref{eq:UPP}. For the family of rescaled measures
	$\hat Q$, see Definition~\ref{def:acc:spine}, this reads
	\begin{equation*}
		\hat Q_x^{k,t}(\Delta F)=\frac{1}{2N^c}
		\int_0^t f(t-s)\int_0^L h(y)\bar{q}_{t-s}(x,y)\hat Q_y^{n,s}(\Delta \psi_0)
		\hat Q_y^{k-n,s}(\Delta \psi_1)dy ds.
	\end{equation*}
\end{proposition}
\begin{cor}[Many-to-few formula for the planarised BBM]
	Let $k\in\mathbb{N}$, $t>0$ and $x>0$. For every bounded measurable function
	$F:\mathbb{U}_k^*\to \R$,
	\[R_x^{k,t}(F)=t^{k-1}Q_x^{k,t}(\Delta F).
	\]
\end{cor}
\begin{proof}
	This result follows from Proposition~\ref{lem:23455} and
	Proposition~\ref{prop:rec_spine} together with the fact that
	the set of functionals of the product form is separating for $\mathbb{U}_k^*$.
\end{proof}
\begin{proof}[Proof of Theorem~\ref{many-to-few00}] The result stems from the above
	corollary by permuting the entries of the marked distance matrix, i.e.~by considering
	functionals $F$ of the form
	\[\forall U=((U_{i,j}),(x_i))\in \mathbb{U}_k^*, \quad F(U) =   \frac{1}{k!}
		\sum_{\sigma \in S_k} \phi((U_{\sigma(i),\sigma(j)}))
		\prod_{i} f_{i}(x_{{\sigma(i)}}).\]
\end{proof}

\section{Heat kernel estimates} \label{sec:hk}
In this section, we derive uniform estimates for the particle density $p_t^\beta\equiv p_t$ defined in Section~\ref{sec:model} and the associated Green's function. For simplicity, we consider similar conventions to that used for the BBMs $\mathbf{X}^N$ and drop the subscripts $\beta$.
All bounds are stated uniformly in the drift parameter $\beta \in (0,1)$, with $L = L_\beta$ chosen as in \eqref{eq:def_L}.

We refer the reader to \cite[Lemma~2.1]{powell19} for similar estimated in the case of general diffusions in
bounded domains of $\mathbb{R}^d$ and \cite[Lemma 5]{berestycki13} for the case
where $0$ and $L$ are both absorbing.

\subsection{Spectral theory and mixing time}\label{sec:spectral_theory}
First, we recall classical results from Sturm--Liouville theory following
\cite[Section~4.6]{zettl10}. Consider the Sturm--Liouville problem
\begin{equation}\tag{SLP}
	\frac{1}{2}v''(x)=\lambda v(x), \quad x\in(0,L), \quad v'(0)=\beta v(0), \quad v(L)=0.
\end{equation}
\begin{itemize}
	\item[(1)] A solution of \eqref{SLP} is defined as a function $v:[0,L]\to \mathbb{R}$ such that $v$ and $v'$ are absolutely continuous on $[0,L]$ and satisfies \eqref{SLP} a.e.~on $(0,L)$. In particular, any solution $v$ is continuously differentiable on $[0,L]$. In our particular case, the solutions are also twice differentiable on $[0,L]$ and satisfy (\ref{SLP}) for all $x\in(0,L)$.
	\item[(2)] A complex number $\lambda$ is an eigenvalue of the Sturm--Liouville problem \eqref{SLP}  if  \eqref{SLP} has a solution $v$ which is not identically zero on $[0,L]$. This set of eigenvalues will be referred to as the spectrum.
	\item[(3)] It is known  that the set of eigenvalues is infinite, countable and has no finite accumulation point. Besides, it is upper bounded and all the eigenvalues are simple and real so that they can be enumerated
	      \begin{equation*}
		      \lambda_1>\lambda_2>...> \lambda_k>... \, ,
	      \end{equation*}
	      where
	      \begin{equation*}
		      \lambda_k\rightarrow -\infty \quad \textnormal{ as } \quad  k\rightarrow+\infty.
	      \end{equation*}
	\item[(4)] As a consequence, the eigenvector $v_i$ associated to $\lambda_i$ is unique up to constant multiplies. Furthermore, the sequence of eigenfunctions can be normalised to form an orthonormal sequence of $\mathrm{L}^2([0,L])$. This orthonormal sequence is complete in $\mathrm{L}^2([0,L])$ so that the fundamental solution of PDE \eqref{PDE:B} can be written as
	      \begin{equation}
		      g_t(x,y)=\sum_{k=1}^\infty e^{\lambda_k t}\frac{v_k(x)v_k(y)}{\|v_k\|^2}.\label{def:qt1}
	      \end{equation}
	\item[(5)] The function $v_1$ does not change sign in $(0,L)$.
		      {For $k\geq 2$, the eigenfunction $v_k$ has exactly $k-1$ zeros on $(0,L)$.}
\end{itemize}
Finally, we recall from \eqref{def:q} that the density of the BBM $\mathbf{X}^\beta$ is related to $g_t$ via the relation
\begin{equation}\label{eq_pg}
	p(x,y)=
	e^{\frac{\gamma^2}{2}t}e^{\beta (y-x)}g_t(x,y), \quad \forall \, x,y \in [0,L],
\end{equation}
where $\gamma$ is as in \eqref{eq:def_gamma}.
\begin{lemma}\label{lem:decomp_spect}
	For all $t>0$, and all $x,y\in [0,L]$,
	\begin{equation*}
		p_t(x,y)=e^{\frac{1-\beta^2}{2}t}e^{\beta(y-x)}\sum_{k\geq1}
		e^{-\frac{\gamma_k^2}{2}t}\frac{v_k(x)v_k(y)}{\|v_k\|^2},
	\end{equation*}
	where $\gamma_1=\gamma =\sqrt{1-\beta^2}$ and, for all $k\in\mathbb{N}$,  $\gamma_k$ is the unique solution of
	\begin{equation*}
		\tan(\gamma_k L)=-\frac{\gamma_k}{\beta}, \quad \text{such that} \quad \gamma_kL\in\left[\left(k-\frac{1}{2}\right)\pi,k\pi\right],
	\end{equation*}
	and
	\begin{equation*}
		v_k(x)=\sin(\gamma_k(L-x)), \; x\in[0,L(\beta)], \quad \|v_k\|^2=\frac{1}{2}\left(L+\frac{1}{\beta}\cos(\gamma_kL)^2\right).
	\end{equation*}
	Moreover, we have
	\begin{equation}
		\sin(\gamma L)=\gamma, \qquad \cos(\gamma L)=-\beta, \qquad  \|v_1\|^2=\frac{L+\beta}{2}. \label{eq:sin/cos}
	\end{equation}
\end{lemma}
\begin{proof}
	The formula for $p_t$ follows directly from points (1) to (6) and from a straightforward calculation.
	The last part of the result follows from the observation that
	\begin{equation*}
		\tan(\gamma L)=-\frac{\gamma}{\beta}, \quad \text{and}\quad  \beta^2+\gamma^2=1.
	\end{equation*}
\end{proof}

\begin{lemma}
	\label{lem:first_eigenv}
	There exists $\beta_0\in(1/2,1)$ such that, for all $\beta> \beta_0$, we have
	\begin{equation*}
		\begin{cases}
			\gamma \left( \frac{1}{2\pi}x+1\right)
			\leq v_1(x)\leq \gamma (\beta x+1), & 0\leq x
			\leq L-\frac{\pi}{2\gamma}                                                \\
			\frac{2\gamma}{\pi}(L-x)\leq v_1(x) \leq \gamma (L-x),
			                                    & L-\frac{\pi}{2\gamma}\leq x \leq L,
		\end{cases}
	\end{equation*}
	For $k\geq 1$, we have
	\begin{equation*}
		|v_k(x)|\leq
		\begin{cases}
			\gamma_k(x+ 2) & 0\leq x
			\leq L-\frac{\pi}{2\gamma}                          \\
			\gamma_k (L-x) & L-\frac{\pi}{2\gamma}\leq x\leq L.
		\end{cases}
	\end{equation*}

\end{lemma}
\begin{proof} Using a convexity argument combined with \eqref{eq:sin/cos},
	one can show that
	\begin{equation*}
		\begin{cases}
			\gamma \left( \frac{2(1-\gamma)}{2\gamma L-\pi}x+1\right)
			\leq v_1(x)\leq \gamma (\beta x+1), & 0\leq x
			\leq L-\frac{\pi}{2\gamma}                                                \\
			\frac{2\gamma}{\pi}(L-x)\leq v_1(x) \leq \gamma (L-x),
			                                    & L-\frac{\pi}{2\gamma}\leq x \leq L.
		\end{cases}
	\end{equation*}
	Then, Remark~\ref{rem:bij_L} shows that, for $L$ large enough,
	$2\gamma L-\pi\leq 2\pi$ and $2(1-\gamma)\geq 1$.
	Similarly, for $k\geq 1$, we see that
	\begin{equation*}
		|v_k(x)|\leq
		\begin{cases}
			\beta|\sin(\gamma_kL)|x+\sin(\gamma_k L) & 0\leq x
			\leq L-\frac{\pi}{2\gamma}                                                    \\
			\gamma_k (L-x)                           & L-\frac{\pi}{2\gamma}\leq x\leq L.
		\end{cases}
	\end{equation*}
	We then recall from Lemma~\ref{lem:decomp_spect} that
	\begin{equation}
		|\sin(\gamma_k L)|= \frac{\gamma_k}{\beta}|\cos(\gamma_k L)|\le
		\frac{\gamma_k}{\beta},
	\end{equation}
	and use that $\beta_0>1/2$.
\end{proof}
\begin{cor}\label{cor:est_v1}

	There exists positive constants $c,c'>0$ such that, for all $\beta\in (\beta_0,1)$ and all $x\in (0,L)$, we have
	\[
		\gamma^2(L-x)\leq c v_1(x) \quad \text{and}
		\quad
		\gamma^2(\beta x+1)(1\wedge (L-x))\leq c' v_1(x).
	\]
\end{cor}

\begin{lemma}\label{lem:hk}
	For all $\beta\in (\beta_0,1)$, all $t>0$ and all $x,y\in [0,L]$,
	\begin{equation*}
		|p_t(x,y)-h(x)\tilde{h}(y)|\leq
		4e^{-\frac{\pi^2}{L^2}t}\left(\sum_{k\geq 2}k^2
		e^{-\frac{(k^2-4)\pi^2}{L^2}t}
		\right)h(x)\tilde{h}(y).
	\end{equation*}
	In particular, for all $\vep>0$, there exists $c_{\ref{lem:hk}}=c_{\ref{lem:hk}}(\vep)>0$
	such that
	\begin{equation*}
		\forall\; \beta\in (\beta_0,1), \quad   t>c_{\ref{lem:hk}} {L}^2, \qquad \quad \forall \;
		x,y\in[0,L],\quad |p_t(x,y)-h(x)\tilde{h}(y)|\leq
		\vep h(x)\tilde{h}(y),
	\end{equation*}
	or equivalently that,
	\begin{equation*}
		\forall\; \beta\in (\beta_0,1), \quad   t>c_{\ref{lem:hk}} {L}^2, \qquad \quad \forall \;
		x,y\in[0,L],\quad |q_t(x,y)-h(y)\tilde{h}(y)|\leq
		\vep h(y)\tilde{h}(y),
	\end{equation*}
\end{lemma}

\begin{proof}
	First, note that $p_t$ can be written as
	\begin{equation*}
		p_t(x,y)=h(x)\tilde h(y)+e^{\beta(y-x)}\sum_{k\geq 2}
		e^{\frac{1}{2}(\gamma_1^2-\gamma_k^2)t}\frac{v_k(x)v_k(y)}{\|v_k\|^2}.
	\end{equation*}
	Then, remark that for all $k\geq 2$,
	\begin{equation*}
		\frac{\|v_1\|^2}{\|v_k\|^2}\leq \frac{1}{1+\frac{1}{\beta L}}\leq 1,
		\quad \text{and} \quad \gamma^2-\gamma_k^2\leq -\frac{k^2\pi^2}{2L^2}.
	\end{equation*}
	For $x\in\left[L-\frac{\pi}{2\gamma},L\right]$, we see from
	Lemma~\ref{lem:decomp_spect} and Lemma~\ref{lem:first_eigenv} that
	\begin{equation*}
		\left|\frac{v_k(x)}{v_1(x)}\right|\leq
		\frac{\pi}{2}\frac{\gamma_k}{\gamma}
		\leq k\pi.
	\end{equation*}
	Similarly, for $x\in\left[0,L-\frac{\pi}{2\gamma}\right]$, we get that
	\begin{equation*}
		\left|\frac{v_k(x)}{v_1(x)}\right|
		\leq \left(\frac{\gamma}{\gamma_k}\right)\frac{\frac{1}{2\pi} x+1}{x+2}
		\leq 2k.
	\end{equation*}
	Putting all of this together, we finally get that
	\begin{equation*}
		e^{\beta(y-x)}\sum_{k\geq 2}e^{\frac{1}{2}(\gamma_1^2-\gamma_k^2)t}
		\frac{|v_k(x)v_k(y)|}{\|v_k\|^2}\leq
		\left(4\sum_{k\geq 2}k^2e^{-\frac{k^2\pi^2}{4L^2}t}\right)h(x)\tilde h(y),
	\end{equation*}
	which concludes the proof of the lemma.

\end{proof}

\subsection{The Green's function}
In this section, we derive the Green's function associated to the
kernel $q_t$ to control the small-time fluctuations of the reflected BBM.
Recall from Section~\ref{sec:model} that $g^\beta\equiv g$ from \eqref{def:q} refers to
the unique fundamental solution of the PDE \eqref{PDE:B}.
Let $B_t$ be a reflecting Brownian motion on $[0,L]$ killed elastically at $0$
with killing coefficient $\beta$ and absorbed at $L$. Its transition kernel is $g$ and its
generator $\mathcal A$ is given
by
\begin{eqnarray*}
	&\mathcal{A}f(y)=\frac{1}{2}f''(y), \quad \text{on the domain}
	\quad \mathcal{D}=\{f: \; f,\mathcal{A}f\in \mathcal{C}([0,L]), \;
	f'(0^+)=\beta f(0^+), \; f(L)=0\}.
\end{eqnarray*}
The associated Green's function $G$
is defined as the unique function such that,
for every bounded measurable functions $f$, we have
\begin{equation*}
	\mathbb{E}_x\left[\int_0^\tau f(B_t)dt\right]=\int_0^\infty G(x,y)f(y)dy,
\end{equation*} with
$\tau:=\inf\{t>0:B_t\notin[0,L )]\}$.
In particular, we have
\begin{equation}
	\label{eq:green_def}
	\int_0^\infty g_s(x,y)ds=G(x,y).
\end{equation}
We know from Section~\ref{sec:spectral_theory} that $v_1\geq 0$
and $\mathcal{A} v_1\leq 0$ on $[0,L]$.
This implies (see \cite[Proposition 4.2.3]{pinsky1995positive}) that
the Green's function $G$ is finite (i.e.~the operator $\mathcal{A}$ is
subcritical in the sense of \cite[Section 4.3]{pinsky1995positive}).
The next
result gives an explicit formula for the Green function $G$.
\begin{lemma}[\textnormal{\cite[p.19]{borodin2015handbook}}]
	\label{lem:green} For all $\beta\in (0,1)$, we have
	\begin{equation*}
		G(x,y)=\begin{cases}
			(\beta x+1)  (L -y)/
			(\beta L +1) & 0\leq x\leq y\leq L   \\
			(\beta y+1)(L -x)/(\beta
			L +1)        & 0\leq y\leq x\leq L .
		\end{cases}
	\end{equation*}
\end{lemma}

In the next lemma, we use this explicit formula to control
the variance in the number of particles generated within a time interval of length $c_{\ref{lem:hk}} {L}^2$.

\begin{lemma}
	\label{lem:short_time_variance}
	Let $\vep>0$ and $T_0=c_{\ref{lem:hk}}(\vep) L^2$.
	There exists a positive constant $c$ such that, for all $\beta \in (\beta_0,1)$ and all $x\in(0,L)$, we have
	\begin{equation*}
		\int_0^{T_0}\left(\int_0^L h(y)q_s(x,y) dy \right)ds \leq c
		\gamma^3e^{\beta L}.
	\end{equation*}
\end{lemma}

The proof of this lemma relies on the following observation.
\begin{rem}
	We see from Remark~\ref{rem:bij_L} that for all $\vep>0$, there exists a constant $c_\vep>0$ such that, for all $\beta\in (\beta_0,1)$ and all $t\leq c_{\ref{lem:hk}}(\vep)L^2$,
	\begin{equation}
		\label{eq:ub_growth}
		\exp\left(\frac{1-\beta^2}{2}t\right)
		\leq c_\vep.
	\end{equation}

\end{rem}

\begin{proof}[Proof of Lemma~\ref{lem:short_time_variance}]
	By \eqref{def:q}, \eqref{eq:rel_qp} and \eqref{eq:ub_growth},
	there exists a constant $C>0$ such that, for all $\beta\in (\beta_0,1)$ and all $x,y\in (0,L)$,
	\begin{equation*}
		\int_0^{T_0} q_s(x,y)ds
		\leq C\frac{v_1(y)}{v_1(x)}\int_0^\infty g_s(x,y) ds\leq
		C\frac{v_1(y)}{v_1(x)}G(x,y).
	\end{equation*}
	By Fubini's theorem, we then see that
	\begin{align*}
		v_1(x)\int_0^{T_0}\left(\int_0^L h(y)q_s(x,y) dy \right)ds &
		\leq C\int_0^L h(y)v_1(y)G(x,y)dy.
	\end{align*}
	Define $I_x:=\int_0^L h(y)v_1(y)G(x,y)dy$.
	It follows from Lemma~\ref{lem:green} and \eqref{def:h} that
	\[
		I_x=\frac{2}{L+\beta}\left(
		\frac{L-x}{\beta L+1}
		I_{x,1}
		+\frac{\beta x+1}{\beta L+1}
		I_{x,2}\right),
	\]
	with $I_{x,1}:=\int_0^xe^{\beta (L-y)}v_1(y)^2\gamma(\beta y+1)dy$
	and $I_{x,2}:=\int_x^L e^{\beta (L-y)}v_1(y)^2\gamma(L-y)dy$.
	By a straightforward convexity argument (similar to that used
	in the proof of ~\cref{lem:first_eigenv}), we see that there exists a constant $C$ such that, for all $\beta\in (\beta_0,1)$ and all $x\in (0,L)$,
	\[v_1(x)\leq C\gamma [(\beta x+1)\wedge(L-x)]\]
	This yields
	\[
		I_{x,1}\leq \gamma^3 \int_0^x e^{\beta (L-y)}(\beta y+1)^3dy\leq
		C\gamma^3e^{\beta L},
	\]
	and
	\[
		I_{x,2} \leq \gamma^3 \int_x^L e^{\beta(L-y)}(L-y)^3 dy
		\leq C \gamma^3e^{\beta L}(1\wedge (L-x)).
	\]
	The result then follows from \cref{cor:est_v1} and \eqref{eq:equiv_L}.
\end{proof}

\section{Convergence of moments}
\label{sec:cv:mom}

We now focus on the sequence of BBMs $(\mathbf{X}^N)_{N\geq N_0}$ defined in~\eqref{eq:def_beta}.
Without loss of generality, we assume that $\beta_N > \beta_0$ for all $N \geq N_0$
(see Lemma~\ref{lem:first_eigenv}), so that the estimates established in Section~\ref{sec:hk} remains valid throughout.

This section is devoted to the proof of Proposition~\ref{th:k-spine-cv}.
As outlined in Section~\ref{sec:heuristics}, this convergence result
relies on the fact that, with high probability, the accelerated spine
is close to equilibrium at the branching points of the spine-tree.
Indeed,  Lemma~\ref{lem:hk} shows  that, for any $\vep >0$, for all $N\geq N_0$,
the `distance' between the distribution of the accelerated spine process $\bar{q}_t$
and its stationary distribution $\tilde h h$ is bounded by $\vep \tilde h h$, for
all $t\geq t_0$, where
\begin{equation}
	\label{eq:def_t0}
	t_0:= \frac{1}{N^{1-c}}c_{\ref{lem:hk}}(\vep) \lcn^2.
\end{equation}
As a consequence, \eqref{eq:mixing} provides a good proxy of the $k$-th moment
whenever the branching times in the $k$-spine tree are separated by at least $t_0$, and lie at least $t_0$ away from both the root and the leaves of the tree.

Recall the definition of the $(U_i)$'s in (see~\eqref{eq:UPP}), and define the `good' event
\begin{equation}
	\label{eq:set_A_def}
	\mathcal{A}:=\{ \forall i\neq j, \; |U_i-U_j|>t_0\}\cap
	\{\forall i, \ 	(U_i\wedge (t-U_i))>t_0 \}.
\end{equation}
On $\mathcal{A}$, we say that $k$-spine tree of depth $t$ has no accumulation  of branching times.
To prove that \eqref{eq:mixing} holds in the large-$N$ limit, we show that spine trees
in $\mathcal{A}^c$, i.e.~those with an accumulation of branching points,
have a negligible contribution to the $k$-th moment.

To establish this result, we will rely on the following direct consequence of
Lemma~\ref{lem:short_time_variance}.  As we shall see, this estimate ensures that variance in the rescaled number of particles produced
before the accelerated spine reaches stationarity is small compared to the typical
population size $N$.

\begin{cor}
	\label{cor:var}
	Let $\vep>0$ and let $t_0$ be as defined in \eqref{eq:def_t0}.
	There exists a constant $c_{\ref{cor:var}}\equiv c_{\ref{cor:var}}(\vep)$
	such that,
	for all $N\geq N_0$ and all $x\in (0,L])$,
	\begin{equation}
		\label{eq:Green_q}
		V(x,t_0):=\frac{1}{N^c}
		\int_0^{t_0}\left(\int_0^{L}\bar{q}_{s}(x,y)h(y)dy\right)ds\leq
		c_{\ref{cor:var}}\frac{\log(N)^3}{N^{1-c}}.
	\end{equation}
\end{cor}

In Section~\ref{sec:rb}, we derive rough bounds on the moments, that will be used
as integrable hats in the proof of Proposition~\ref{th:k-spine-cv}. Section~\ref{sec:cvmom}
is dedicated to the proof of Proposition~\ref{th:k-spine-cv}

\subsection{Integrable bounds}
\label{sec:rb}
\begin{lemma}
	\label{lem:rb_mom1}
	Let $\vep\in(0,1]$ and $t_0$ be as in \eqref{eq:def_t0}.
	There exist a positive constant $c_{\ref{lem:rb_mom1}}>0$ and  $N_{\ref{lem:rb_mom1}}\in \N$ such that, for all $N\geq N_{\ref{lem:rb_mom1}}$,
	the following hold: for all $x\in(0,L)$,
	\begin{equation}
		\label{eq:rb_mom1_eq}
		\forall t> t_0, \quad
		\hat{Q}^{1,t}_x(\Delta)\leq c_{\ref{lem:rb_mom1}},
	\end{equation}
	and,
	\begin{equation}
		\label{eq:rb_mom1_neq}
		\forall t \leq t_0,\quad  \hat{Q}^{1,t}_x(\Delta) \leq \frac{c_{\ref{lem:rb_mom1}}}{h(x)}.
	\end{equation}
	Moreover, for all $N\geq N_{\ref{lem:rb_mom1}}$ and all $x\in(0,L)$,
	\begin{equation}
		\label{eq:rb_mom1_int}
		\int_0^{t_0} \hat{Q}^{1,s}_x(\Delta)\, ds \leq c_{\ref{lem:rb_mom1}}\, \frac{\log(N)^5}{N^{1-c}}.
	\end{equation}
\end{lemma}

\begin{proof}
	By definition of the family of measures $\hat Q$, we have
	\begin{equation}
		\label{eq:q1}
		\hat{Q}^{1,t}_x(\Delta)=\int_0^L \frac{1}{h(y)}\bar{q}_t(x,y)dy.
	\end{equation}
	For $t>t_0$, we use Lemma~\ref{lem:hk} and \eqref{eq:normalisation} to see that
	\[
		\hat{Q}^{1,t}_x(\Delta)\leq (1+\vep)\int_0^L \tilde h(y)dy = 1+\vep.
	\]
	For $t\leq t_0$, we use
	\eqref{eq:rel_qp} and \eqref{eq:ub_growth} to obtain
	\[
		\hat{Q}^{1,t}_x(\Delta)=\int_0^L \frac{1}{h(y)}\bar{q}_t(x,y)dy=
		\frac{1}{h(x)}\int_{0}^{L} \bar{p}_s(x,y)dy\leq \frac{C}{h(x)}.
	\]
	Finally, the third point follows from Fubini's
	theorem, \eqref{eq:green_def} and \eqref{eq:ub_growth},
	\begin{align*}
		\int_0^{t_0}\hat{Q}^{1,s}_x(\Delta)ds & =\frac{1}{h(x)}\int_0^L\left(\int_0^{t_0}
		\bar{p}_{s}(x,y) ds\right)dy\leq  \frac{C}{N^{1-c}}\frac{1}{h(x)}e^{-\beta x}
		\int_0^L e^{\beta y}G(x,y)dy.
	\end{align*}
	Using the explicit formula for $G$ from Lemma~\ref{lem:green},
	the RHS of the above equation can be rewritten as
	\begin{align*}
		\int_0^L e^{\beta y}G(x,y)dy & =\frac{\lcn-x}{\beta\lcn +1}
		\int_0^x e^{\beta y}(\beta y+1)dy+\frac{\beta x+1}{\beta\lcn +1}
		\int_x^\lcn e^{\beta y} (\lcn-y)dy.
	\end{align*}
	\cref{cor:est_v1} then yields
	\[
		\int_0^L e^{\beta y}G(x,y)dy \leq 2L^2e^{\beta L}(L-x)\leq c\gamma^{-4}
		e^{\beta L}v_1(x).
	\]
	Finally, by the definition of $h$ and \eqref{eq:equiv_L}, we have that
	for all $N$ large enough,
	\begin{equation*}
		e^{-\beta x}\int_0^L e^{\beta y}G(x,y)dy\leq 2\gamma^{-4} e^{\beta (L-x)}
		v_1(x)
		\leq C\gamma^{-5}h(x).
	\end{equation*}
	This concludes the proof of the lemma.
\end{proof}

\begin{lemma}[Uniform bounds]
	\label{lem:rb} Let $T>0$ and $k\geq 2$.
	There exist a constant $c_{\ref{lem:rb}}>0$
	and an index $N_{\ref{lem:rb}}\in \N$
	such that, for all $t\in [0,T]$,
	\begin{equation}
		\sup_{N\geq N_{\ref{lem:rb}}}\hat{Q}_x^{k,t,N}\left(\Delta\right)<c_{\ref{lem:rb}}
	\end{equation}
\end{lemma}

\begin{proof}
	We prove this result by induction.
	For $k=2$, we know from Proposition~\ref{prop:rec_spine} that
	\[
		\hat{Q}_x^{2,t}(\Delta)= \frac{1}{2N^{c}}\int_0^t \int_0^L h(y)
		\bar q_s(x,y)\hat{Q}^{1,t-s}_y(\Delta)^2\ dy \ ds.
	\]
	For $t>2t_0$, we split the time integral into three parts: $I_1$,
	where the branching point is within $t_0$ of the root ($s\leq t_0$); $I_2$, where it is within $t_0$ of the leaves ($s\geq t-t_0$); and $I_3$, where it is at least $t_0$ away from both the root and the leaves ($s\in (t_0,t-t_0)$).

	Intuitively, in $I_1$, the two leaves are at equilibrium,
	whereas the branching point is not.
	In this case, we first apply \eqref{eq:rb_mom1_eq}
	to see that, for $N$ large enough,
	\begin{equation*}
		I_1
		\leq \frac{(c_{\ref{lem:rb_mom1}})^2}{2 N^{c} } \int_0^{t_0}\int_0^L h(y)\bar{q}_s(x,y)dy \ ds.
	\end{equation*}
	We then deduce from \eqref{eq:Green_q} that
	\begin{equation*}
		I_1
		\leq \frac{(c_{\ref{lem:rb_mom1}})^2c_{\ref{cor:var}}}{2}\frac{\log(N)^3}{N^{1-c}}.
	\end{equation*}
	In $I_2$, the branching point is close to equilibrium but not the
	leaves. We use Lemma~\ref{lem:hk}, \eqref{eq:rb_mom1_neq}
	and \eqref{eq:normalisation}
	to see that, for $N$ large enough,
	\begin{align*}
		I_2 \leq 2N^{-c}\int_{t-t_0}^t \int_{0}^{L}h(y)^2 \tilde{h}(y)
		\hat{Q}^{1, t-s}_y(\Delta)^2 dy \ ds \leq 4 \frac{t_0}{N^c}\int_{0}^L
		\tilde h(y)dy \leq  4 \frac{t_0}{N^c}. \label{eq:I2}
	\end{align*}
	In $I_3$, both the branching point and the leaves are close to equilibrium.
	Hence, it follows from~\eqref{eq:rb_mom1_eq},  Lemma~\ref{lem:hk} and
	Proposition~\ref{prop:concentration} that, for $N$ large enough,
	\[
		I_3\leq (c_{\ref{lem:rb_mom1}})^2N^{-c}\int_{t_0}^{t-t_0}\int_0^L h(y)^2
		\tilde{h}(y)dy  \ ds
		\leq   2(c_{\ref{lem:rb_mom1}})^2t .
	\]

	For $t\leq2t_0$, we need to deal with the case where neither the branching
	point nor the leaves are at equilibrium, that is, for $s\in[t-t_0,t_0]$.
	We denote this integral by $I_4$.
	In this case, we use~\eqref{eq:rb_mom1_int} and \eqref{eq:q1} to see that
	\[
		I_4\leq \frac{2}{N^c}\int_{t-t_0}^{t_0} \int_0^L \frac{1}{h(y)}\bar{q}_s(x,y)dy \ ds
		\leq \frac{2}{N^c}\int_{0}^{t_0}\hat{Q}^{1,s}_x(\Delta)ds\leq 2c_{\ref{lem:rb_mom1}} \frac{\log(N)^5}{N}.
	\]
	The remaining parts of the time integral are then dealt with as in the case $t>2t_0$.
	This concludes the proof for $k=2$.

	For $k\geq 3$, we use the fact that $\hat{Q}_x^{k,t}$ satisfies the recursive relation
	\begin{equation}
		\label{eq:rec_hatQ}
		\hat{Q}^{k,t}_x(\Delta )=\frac{1}{2N^c}\int_0^t \int_0^L h(y)\bar{q}_{t-s}(x,y)
		\sum_{n=1}^{k-1}\hat Q^{n,s}_y (\Delta )\hat Q^{k-n,s}_y (\Delta )dyds.
	\end{equation}
	This formula is obtained from Proposition~\ref{prop:rec_spine} by
	and summing over all the
	possible sizes for the left and the right subtree in the $k$-spine tree.
	By the induction hypothesis, we then get that
	\begin{multline*}
		\hat{Q}^{k,t}_x(\Delta )\leq
		\frac{1}{N^c} \left(\left[\sum_{n=2}^{k-2}c_{\ref{lem:rb}}(n,T)c_{\ref{lem:rb}}(k-n,T)\right]\int_0^t \int_0^L h(y)\bar{q}_s(x,y)dy \ ds \right.\\
		\left.	+	2c_{\ref{lem:rb_mom1}}c_{\ref{lem:rb}}(k-1,T)\int_0^t\int_0^Lh(y)\bar{q}_s(x,y) \hat{Q}_y^{1,s}dy\ ds.
		\right)
	\end{multline*}
	The first term can be bounded by a constant using Lemma~\ref{lem:short_time_variance}
	and Proposition~\ref{prop:concentration} together with Lemma~\ref{lem:hk}. The second
	term is handled as in the case $k=2$, using Lemma~\ref{lem:hk}, Lemma~\ref{lem:rb_mom1}
	and Proposition~\ref{prop:concentration}.
\end{proof}

\subsection{Proof of Proposition~\ref{th:k-spine-cv}}
\label{sec:cvmom}
In this section, we formally justify the approximation \eqref{eq:mixing}. We first
prove that the contribution of k-spine trees comprising accumulations of
branching points is small.
\begin{lemma}
	\label{lem:accumulation}
	Let $t>0$ and $k\in \N$. Let $\mathcal{A}$ be as in \eqref{eq:set_A_def}.
	We have
	\[
		\hat{Q}_x^{k,t}(\Delta \mathbf{1}_{\mathcal{A}^c})\to 0,\quad \text{as} \quad N\to\infty.
	\]
\end{lemma}
\begin{proof}
	We proceed by induction on $k$.
	For $k=2$, the $2$-spine tree has a single branching point, and the event $\mathcal{A}$ corresponds to this branching point being within a distance $t_0$ of either the root or the leaves. Thus,
	\[
		\hat{Q}_x^{2,t}\left(  \Delta
		\mathbf{1}_{\mathcal{A}^c}\right) = I_1 + I_2,
	\]
	where $I_1$ and $I_2$ are as defined in the proof of Lemma~\ref{lem:rb}.
	This proves the case $k=2$.

	Now, assume the result holds up to rank $ k-1$. For $k \geq 3$, we use the
	recursive structure of the $k$-spine tree.
	Recall the definition of $\tau$ from \eqref{eq:def_tau} and define the event
	\[\mathcal{B}=\{\tau>t-t_0\},\]
	which  corresponds to the event in which the deepest branching
	point of the $k$-spine tree lies within time $t_0$ of the root.
	Note that, by definition of $\mathcal{A}$, we have $\mathcal{B}\subset \mathcal{A}^c$.
	As a consequence,
	\[
		\hat{Q}_x^{k,t}(\Delta \mathbf{1}_{\mathcal{A}^c})\leq \hat{Q}_x^{k,t}(\Delta\mathbf{1}_\mathcal{B})
		+\hat{Q}_x^{k,t}( \Delta\mathbf{1}_{\mathcal{B}^c\cap \mathcal{A}^c}).
	\]
	Moreover, on $\mathcal{B}^c\cap \mathcal{A}^c$, there is an accumulation of branching
	points in at least one the two subtrees attached to the deepest branching point.
	Thus,
	\[
		\hat{Q}_x^{k,t}(\Delta \mathbf{1}_\mathcal{A})\leq \hat{Q}_x^{k,t}(\Delta\mathbf{1}_\mathcal{B})
		+\hat{Q}_x^{k,t}(\Delta \mathbf{1}_{\mathcal{B}^c}\mathbf{1}_{\mathcal{A}_0})
		+\hat{Q}_x^{k,t}(\Delta \mathbf{1}_{\mathcal{B}^c}\mathbf{1}_{\mathcal{A}_1}),
	\]
	where $\mathcal{A}_0$ and $\mathcal{A}_1$ denote the accumulation events in
	the left and right subtrees, respectively.
	Moreover, one can check (as in \eqref{eq:rec_hatQ}) that
	\[
		\hat{Q}^{k,t}_x(\Delta \mathbf{1}_\mathcal{B})=\frac{1}{2N^c}\int_{0}^{t_0}
		\int_0^L h(y)\bar{q}_{s}(x,y) \left[\sum_{j=1}^{k-1}
			\hat Q^{j,t-s}_y (\Delta)\hat Q^{k-j,t-s}_y (\Delta)\right]dy \ ds.
	\]
	In the above sum, the terms with indices $j\in\{2,...,k-2\}$ can be bounded
	using Lemma~\ref{lem:rb}, while the convergence to of the corresponding integral
	follows from Lemma~\ref{lem:short_time_variance}. Hence, all these terms
	converge to $0$. For $j \in \{1,k-1\}$, we need
	to prove that
	\[
		\frac{1}{2N^c}\int_{0}^{t_0}
		\int_0^L h(y)\bar{q}_{s}(x,y)
		\hat Q^{1,t-s}_y (\Delta) dy \ ds\xrightarrow{N\to\infty } 0.
	\]
	The proof of this fact is similar to that of Lemma~\ref{lem:rb}
	and is left to the reader.

	Next, we have
	\[
		\hat{Q}^{k,t}_x(\Delta \mathbf{1}_{\mathcal{B}^c}\mathbf{1}_{\mathcal{A}_0})
		=\frac{1}{2N^c}\int_{t_0}^{t}
		\int_0^L h(y)\bar{q}_{s}(x,y) \left[\sum_{j=1}^{k-1}
			\hat Q^{j,t-s}_y (\Delta\mathbf{1}_\mathcal{A})\hat Q^{k-j,t-s}_y
			(\Delta)\right]dy \ ds.
	\]
	In this case, the spine process is close to equilibrium at the deepest
	branching point of the $k$-spine tree.
	We then see from Lemma~\ref{lem:hk}
	\[
		\hat{Q}^{k,t}_x(\Delta \mathbf{1}_{\mathcal{B}^c}\mathbf{1}_{\mathcal{A}_0})
		\leq 2\int_0^L h(y)^2\tilde{h}(y)\int_0^{t-t_0}\left[\sum_{j=1}^{k-1}
			\hat Q^{j,s}_y (\Delta\mathbf{1}_\mathcal{A})\hat Q^{k-j,s}_y (\Delta)\right]dsdy.
	\]
	Again, we distinguish the case $j\in \{1,k-1\}$ and $j\notin \{1,k-1\}$.
	In the latter case, we use Proposition~\ref{prop:concentration}, Lemma~\eqref{lem:rb}
	and the induction hypothesis together with the dominated convergence theorem
	to prove that this bounds converges to $0$. For $j=1$,
	by convention $\hat{Q}_{x}^{1,s}(\Delta \mathbf{1}_{\mathcal{A}})=0$.
	For $j=k-1$,  we use Lemma~\ref{lem:rb_mom1} (iii), Lemma~\ref{lem:rb} and
	Proposition~\ref{prop:concentration} for $s<t_0$, and Proposition~\ref{prop:concentration},
	Lemma~\ref{lem:rb_mom1} (i) and
	Lemma~\ref{lem:rb} together with the dominated convergence theorem for $s\geq t_0$.
\end{proof}

\begin{proof}[Proof of Proposition~\ref{th:k-spine-cv}]
	Write
	\begin{equation*}
		\hat{Q}^{k,t}_x\left(  \Delta
		\phi( U_{\sigma(i),\sigma(j)})\right)=
		\hat{Q}^{k,t}_x\left(  \Delta
		\phi(( U_{\sigma(i),\sigma(j)})) \mathbf{1}_{\mathcal{A}}\right)
		+ \hat{Q}^{k,t}\left(  \Delta
		\phi(( U_{\sigma(i),\sigma(j)})) \mathbf{1}_{\mathcal{A}^c}\right)
	\end{equation*}
	We deal with the terms separately. First, we see from Lemma~\ref{lem:hk} that
	\begin{equation*}
		\hat{Q}^{k,t}_x\left(  \Delta
		\phi(( U_{\sigma(i),\sigma(j)})) \mathbf{1}_{\mathcal{A}}\right)
		=(1+O(\vep))^{2k-1}\left(\frac{t}{N^c}\right)^{k-1}
		\left(\int_0^Lh(y)^2\tilde h(y)dy\right)^{k-1}\left(\int_0^L\tilde{h}(y)dy\right)^k
		{\mathbb E}
		\left[ \phi ( U_{\sigma(i),\sigma(j)})\mathbf{1}_{\mathcal{A}} \right].
	\end{equation*}
	Recalling that $\phi$ is bounded and using \eqref{eq:normalisation}
	and Proposition~\ref{prop:concentration}, we then get that, for $N$ large enough,
	\[
		\left|\hat{Q}^{k,t}_x\left(  \Delta
		\phi(( U_{\sigma(i),\sigma(j)})) \mathbf{1}_{\mathcal{A}}\right)
		-\left(\frac{\sigma^2t}{2}\right)^{k-1}{\mathbb E}
		\left[ \phi ( U_{\sigma(i),\sigma(j)})\mathbf{1}_{\mathcal{A}} \right]\right|
		=O(\vep).
	\]
	On the other hand, we have
	\[
		{\mathbb E}
		\left( \phi ( U_{\sigma(i),\sigma(j)})\mathbf{1}_{\mathcal{A}}\right)=
		\E\left( \phi ( U_{\sigma(i),\sigma(j)})\right)-\E\left( \phi ( U_{\sigma(i),\sigma(j)})\mathbf{1}_{\mathcal{A}^c}\right),
	\]
	and
	\[
		\E\left( \phi ( U_{\sigma(i),\sigma(j)})\mathbf{1}_{\mathcal{A}^c}\right)\leq C \P(\mathcal{A}^c)
		\xrightarrow{N\to\infty} 0.
	\]
	We then use that $\phi$ is bounded and apply Lemma~\ref{lem:accumulation} to
	get that
	\[
		\hat{Q}^{k,t}_x\left(  \Delta
		\phi(( U_{\sigma(i),\sigma(j)})) \mathbf{1}_{\mathcal{A}^c}\right)\xrightarrow{N\to \infty} 0,
	\]
	which concludes the proof of the result.
\end{proof}

\section{Convergence of genealogies}\label{sec:cv:genealogies}
In this section, we prove Theorem~\ref{th:genealogy}, assuming that the Kolmogorov estimate stated in Theorem~\ref{th:Kolmogorov} is valid. The proof of Theorem~\ref{th:Kolmogorov} is postponed to the following section.

\begin{proof}[Proof of Theorem~\ref{th:main-theorem}]
	Let $k\in\mathbb{N}$, $t>0$ and $x\in(0,L)$.
	{Let $\phi:[0,\infty)^{\binom{k}{2}}\to \mathbb{R}$ be a bounded measurable function and define
	\[
	\tilde{\Phi}((X,d,\mu))=\int_{(v_i)\in X^k \; \text{distinct} }\phi(d(v_i,v_j)_{1\leq i<j\leq k})
	\prod_{i=1}^{k}\mu(dv_i),
	\]}
	and
	\[
		{\Phi}((X,d,\mu))=\int_{(v_i)\in X^k \; }\phi(d(v_i,v_j)_{1\leq i<j\leq k})
		\prod_{i=1}^{k}\mu(dv_i)
	\]
	It follows from Proposition~\ref{th:many-to-few2} and Proposition~\ref{th:k-spine-cv} that
	\begin{equation*}
		\lim_{N\to\infty}N\mathbb {E}_x\left[\tilde  \Phi(\bar{M}_t) \right] = k! \left( \frac{\sigma^2 t}{2} \right)^k \EE{\phi( U_{\sigma(i),\sigma(j)})},
	\end{equation*}
	where $(U_{i,j})$ is as in \eqref{eq:UPP} and
	$\sigma$ is an independent random permutation of $[k]$.
	Combining this with our Kolmogorov estimate (Theorem~\ref{th:Kolmogorov}), we obtain
	\begin{align*}
		\lim_{N \to \infty} \E_x[\tilde \Phi(\bar{M}_t)\mid Z_{t N^{1-c}}>0] = \lim_{N \to \infty} \frac{N \E_x[\tilde \Phi(\bar{M}_t)]}{N \P_x(Z(t N^{1-c})>0)}
		=k! \left( \frac{\sigma^2 t}{2} \right)^k \EE{\phi( U_{\sigma(i),\sigma(j)})}.\end{align*}
	It now remains to prove that
	\[
		\lim_{N\to\infty}
		\E_x \left[  \tilde \Phi(\bar{M}_t) - \Phi(\bar{M}_t)
			\mid Z_{t N^{1-c}}>0  \right]=0
	\]
	Hence, it is sufficient to show by induction on $k$ that
	\[
		\lim_{N\to\infty}\E_x\left[\int_{\Nc_t^k} \mathbf{1}_{\cup_{1\leq i<j\leq k}\{v_i=v_j\}}
		\prod_{n=1}^k \bar{\mu}
		(dv_n)\big| Z_{tN^{1-c}}>0\right]=0.
	\]
	On the one hand,
	\[
		\E_x\left[\int_{\Nc_t^k} \mathbf{1}_{\cup_{1\leq i<j\leq k}\{v_i=v_j\}}
		\prod_{n=1}^k \bar{\mu}
		(dv_n)\big| Z_{tN^{1-c}}>0\right]\leq \sum_{1\leq i<j\leq k}
		\frac{1}{N}\E_x \left[\prod_{n=1}^{k-1} \bar{\mu}
			(dv_n)\big| Z_{tN^{1-c}}>0\right].
	\]
	On the other hand, by induction and by the first part of the proof,
	\begin{align*}
		\lim_{N\to\infty}\E_x \left[\prod_{n=1}^{k-1} \bar{\mu}
			(dv_n)\big| Z_{tN^{1-c}}>0\right]
		 & =\lim_{N\to\infty}\E_x\left[\int_{\Nc_t^k} \mathbf{1}_{\{(v_i)\ \text{distinct}\}}
		\prod_{n=1}^k \bar{\mu}
		(dv_n)\big| Z_{tN^{1-c}}>0\right]                                                     \\
		 & =(k-1)! \left( \frac{\sigma^2t}{2} \right)^{k-1}.
	\end{align*}
	This concludes the proof of the result.
\end{proof}

\begin{proof}[Proof of Theorem \ref{th:Yaglom} and \ref{th:genealogy}]
	The proof goes along the same lines as the proof of Corollary 2.3 in~\cite{boenkost2022genealogy}
	by noting that the respective maps, mapping the Brownian CPP to its total size and to
	its genealogy, are continuous w.r.t. the Gromov--weak topology and making use of Theorem~\ref{th:main-theorem}.
\end{proof}

\section{Kolmogorov estimate (Proof of Theorem~\ref{th:Kolmogorov})}
\label{sec:Kolmogorov}

The proof of Theorem~\ref{th:Kolmogorov} is an adaptation of \cite[Proposition 4.1]{powell19},
which treats the case of critical branching diffusions in bounded domains.

Following~\cite{powell19}, we define
\begin{align*}
	f(t,x) = \pp_x\big(Z(t)>0\big), \qquad t\ge 0, \quad x\in[0,\lcn],
\end{align*}
and
\begin{align*}
	a(t) = \int_0^{\lcn} f(t,x)\,\tilde h(x)\,dx, \qquad t\ge 0.
\end{align*}
For clarity, we omit the superscript $N$, keeping in mind that all quantities still depend implicitly on $N$.

The main step is to show that, for $N$ and $t$ sufficiently large,
\begin{equation}\label{approx:f:a}
	f(t,x)\approx a(t)h(x),
\end{equation}
and that $a$ satisfies, at least approximately,
\begin{equation}\label{eq:ODE:a}
	\dot a(t)\approx -\frac{\Sigma^2}{2}a(t)^2.
\end{equation}
Solving this ODE and rescaling time then gives the result.
More precisely, we prove the following estimate.

\begin{proposition}[Kolmogorov estimate for the unscaled BBM]\label{prop:kolmogorov}
	For all $\vep>0$, there exists $N^*\in\N$ such that,
	for all $N>N^*$, all $t\ge \log(N)^7$ and all $x\in(0,L)$,
	\[
		(1-\vep)\frac{h(x)}{t}\frac{2}{\Sigma^2}
		\le \P_x\big(Z(t)>0\big)
		\le (1+\vep)\frac{h(x)}{t}\frac{2}{\Sigma^2}.
	\]
\end{proposition}

The proof of this result is an adaptation of the following result obtained by Powell
in~\cite{powell19}.

\begin{proposition}[\cite{powell19}, Proposition 4.1]
	\label{prop:powell}
	For all $\vep>0$ and all $N\ge N_0$, there exists $T(\vep,N)$ such that,
	for all $t>T(\vep,N)$ and all $x\in(0,L)$,
	\[
		(1-\vep)a(t)h(x)\le f(t,x)\le (1+\vep)a(t)h(x).
	\]
\end{proposition}

The proof of Proposition~\ref{prop:powell} relies on a spine construction that differs from
the one introduced in Section~\ref{sec:spine}. For clarity, we do not give the
full alternative construction here. Instead, we start from the bounds obtained through
this spinal decomposition in~\cite{powell19}, and show that they can be extended to the
increasing sequence of domains considered here.
More precisely, we prove that for $N$ sufficiently large, the time $T(\vep,N)$
introduced in Proposition~\ref{prop:powell} can be chosen much smaller than $N^{1-c}$,
namely of order $\log(N)^7$.

Essentially, this alternative spinal decomposition relies on the martingale
change of measure
\begin{align}\label{chg:measure}
	\frac{d \tilde{\P}^{t}_x}{d \P_x} = \frac{1}{h(x_v(t))}Y_t,
\end{align}
where $Y_t$ refers to the additive martingale defined in \eqref{eq:def_martingale}.
Under this new measure, the spine behaves as a distinguished
particle which: (i) evolves according to the transition kernel $q(x,y)$; (ii)
branches at rate $1$; and (iii) at each branching event of a spine particle, one
of the two offspring is chosen to continue the spine and then repeats the behaviour
of its parent. The non-spine particles evolve independently and follow the same
dynamics as in the original branching Brownian motion $\mathbf{X}^N$. We refer
to~\cite[Section 2.3.4]{powell19} and to the references therein
for a precise definition of this change of measure.

\subsection{The function $a$}
We first derive rough bounds on the function $a$, which will play a central role
in the refined estimates of $f$.

\begin{lemma}\label{lem:ODE}
	For all $N\geq N_0$ and all $t>0$,
	\begin{align}
		\label{eq:derivative_a}
		\dot{a}(t) & = - \frac{1}{2}\int_0^{\lcn}
		f(t,x)^2\tilde h(x)dx.
	\end{align}
	As a consequence, we have, for all $N\geq N_0$ and all $t>0$,
	\[
		a(t)\leq \frac{2}{t}.
	\]
\end{lemma}

\begin{proof}
	By definition of $a$, we see that
	\begin{equation*}
		\dot{a}(t)=\int_0^{\lcn} \partial_t f(t,x)\tilde h(x)dx.
	\end{equation*}
	It is a standard fact that $f$ satisfies the FKPP equation
	\begin{equation*}
		\begin{cases}
			\partial_t f(t,x) = \frac{1}{2} \partial_{xx} f(t,x) +
			\beta \partial_x f(t,x) + \frac{1}{2} \left( f(t,x) -f(t,x)^2\right), \\
			f(t,\lcn) =0 , \quad 	\partial_x f(t,x)|_{x=0}=0.
		\end{cases}
	\end{equation*}
	An integration by part then yields
	\begin{align*}
		\dot{a}(t) &
		=\left[ \frac{1}{2} \partial_x f(t,x) \tilde{h}(x) \right]_0^\lcn
		- \frac{1}{2}\int_0^\lcn  \partial_x f(t,x) \tilde{h}'(x) dx
		+ \left[ \beta f(t,x) \tilde{h}(x) \right]_0^\lcn                              \\
		           & \quad- \beta\int_0^\lcn  f(t,x) \tilde{h}'(x) dx
		+ \frac{1}{2}\int_0^\lcn f(t,x)\tilde{h}(x)dx - \frac{1}{2}\int_0^\lcn
		f(t,x)^2\tilde{h} (x)dx                                                        \\
		           & = \left[ -\frac{1}{2} f(t,x) \tilde{h}'(x) \right]_0^\lcn - \beta
		f(t,0) \tilde{h} '(0) + \int_0^\lcn  f(t,x) \left( \frac{1}{2}\tilde{h}''(x)
		-\beta \tilde{h}'(x)-\frac{1}{2}\tilde{h}(x)\right) dx                         \\
		           & \quad
		- \frac{1}{2}\int_0^\lcn f(t,x)^2\tilde{h}(x)dx                                \\
		           & =- \frac{1}{2}\int_0^\lcn f(t,x)^2\tilde{h} (x)dx,
	\end{align*}
	where we use that $\frac{1}{2}\tilde{h}''-\beta \tilde{h}'
		-\frac{1}{2}\tilde{h}$ on $(0,\lcn)$, that  $f(t,\lcn)=0$
	and that $\frac{1}{2}\tilde{h}'(0)-\beta \tilde{h}(0)=0$ to get the last line.

	The second part of the lemma follows from Jensen's inequality.
	Indeed, we have
	\begin{equation*}
		\dot{a}(t)
		= - \frac{1}{2}\int_0^{\lcn}f(t,x)^2\tilde h(x)dx
		\leq - \frac{1}{2}\left(\int_0^{\lcn}f(t,x)\tilde h(x)dx\right)^2
		=- \frac{1}{2}a(t)^2.
	\end{equation*}
	Integrating this inequality then yields that
	\begin{equation*}
		\forall t>0, \quad a(t)\leq\frac{1}{\frac{1}{a(0)}+\frac{1}{2}t}
		\leq \frac{2}{t}.
	\end{equation*}
\end{proof}

\begin{lemma}
	\label{lem:rb_mart}
	Recall the definition of the additive martingale $Y$ from \eqref{eq:def_martingale}.
	There exists a constant $c_{\ref{lem:rb_mart}}>0$ such that, for $N$ large enough,
	for all $t\geq\log(N)^3$
	and all $x\in (0,L)$,
	\[
		\E_x\left[{Y_t}^{2}\right]\leq c_{\ref{lem:rb_mart}}{t}h(x)N^c.
	\]
\end{lemma}
\begin{proof}
	First, we note that
	\[ \E_x\left[{Y_t}^2\right]=\E_x\left[\sum_{v\in \Nc_t}
			h(x_v(t))^2\right]+\E_x\left[\sum_{v\neq w\in \mathcal{N}_t}h(x_{v}(t))h(x_{w}(t))
			\right].
	\]
	By the many-to-few lemma (Theorem~\ref{many-to-few00}), we get that
	\[
		\E_x\left[\sum_{v\neq w\in \mathcal{N}_t}h(x_{v}(t))h(x_{w}(t))
			\right]=
		h(x)tQ_{x}^{2,t}\left( h(\zeta_v)\right),
	\]
	where $v$ refers to the unique branching point of the $2$-spine tree.
	By definition of the spine measure,
	\[
		Q_{x}^{2,t}\left( h(\zeta_v)\right)={\frac{1}{t}}\int_0^t\left(\int_0^L q_s(x,y)h(y)dy \right)ds.
	\]
	Let $T_0=c_{\ref{lem:hk}}(1) \lcn^2$.
	It follows from Lemma~\ref{lem:short_time_variance}
	that, for all $N$ large enough and all $x\in(0,L)$,
	\begin{equation*}
		\int_0^{T_0}\left(\int_0^L q_s(x,y)h(y)dy \right)ds\leq
		c_{\ref{lem:short_time_variance}}\log(N)^3N^c.
	\end{equation*}
	The remaining part of the integral is bounded using Lemma~\ref{lem:hk}:
	we obtain that, for $N$ large enough,
	\begin{equation*}
		\int_{T_0}^t\left(\int_0^L q_s(x,y)h(y)dy \right)ds\leq
		2(t-T_0)\left(\int_0^L h(y)^2\tilde{h}(y)dy \right)ds\leq 4\sigma^2tN^c,
	\end{equation*}
	where we used Proposition~\ref{prop:concentration} to obtain the last
	inequality.

	Next, applying the many-to-one once again, we obtain
	\[
		\E_x\left[\sum_{v\in \Nc_t}
			h(x_v(t))^2\right]=
		h(x)Q_{x}^{1,t}\left( h(\zeta_v)\right),
	\]
	where $v$ refers to the unique leaf of the $1$-spine tree.
	By definition of the spine measure $Q_x^{1,t}$,
	\[
		Q_x^{1,t}\left( h(\zeta_v)\right)
		=\int_0^L q_t(x,y)h(y)dy.
	\]
	It then follows from \cref{lem:hk}
	and Proposition~\ref{prop:concentration} that
	\[
		Q_x^{1,t}\left( h(\zeta_v)\right)\leq 2\int_0^L h(y)^2\tilde{h}(y)dy\leq 4\sigma^2N^c.
	\]
	Putting all of these estimates together yields the result.
\end{proof}

\begin{lemma}[Lower bound on $a$]
	\label{lem:A priori bound}
	There exists a constant $c_{\ref{lem:A priori bound}}>0$ such that, for $N$ large enough and for all $t\geq \log(N)^3$,
	\begin{align}
		\label{rough:est}
		f(t,x)\geq \frac{c_{\ref{lem:A priori bound}}}{ tN^c }h(x).
	\end{align}
	Moreover, there exists a constant $c_{\ref{lem:A priori bound}}'>0$
	such that,
	for $N$ large enough and all $t\geq \log(N)^3$
	\begin{equation*}
		\dot{a}(t)\leq -\frac{c_{\ref{lem:A priori bound}}'}{t^2}\frac{1}{N^c}
		\quad \text{and} \quad a(t)\geq\frac{c_{\ref{lem:A priori bound}}'}{t} \frac{1}{N^c}.
	\end{equation*}
\end{lemma}
\begin{proof}
	The proof is adapted from \cite[Lemmas 7.2]{harris21} and relies on
	a change of measure combined with Jensen's inequality. This change of
	measure is precisely the one used in the alternative spine decomposition
	introduced by Powell in~\cite{powell19}.

	Let $\tilde{\P}_x^{t}$ be the probability measure absolutely continuous
	w.r.t.~to $\mathbb{P}_x$ whose  Radon-Nikodym derivative is  given by
	\eqref{chg:measure}.
	This change of measure combined with Jensen's inequality yields
	\begin{align*}
		\frac{\P_x(Z(t)>0)}{h(x)} =
		\tilde{\P}^{t}_x \left[\frac{1}{\sum_{v \in \Nc_{t}} h(x_v(t))}\right]
		\geq \frac{1}{\tilde{\P}^{t}_x \left[ \sum_{v \in \Nc_{t}} h(x_v(t))\right]}
	\end{align*}
	Yet, we see from \eqref{chg:measure} that
	\begin{align}
		\tilde{\P}^{t}_x \left[ \sum_{v \in \Nc_{t}} h(x_v(t))\right]
		=\frac{{\E}_x \left[ \left( \sum_{v \in \Nc_{t}} h(x_v(t))\right)^2\right]}{h(x)}.
	\end{align}
	The first part of the result then follows from \cref{lem:rb_mart}.

	The second part of the result stems from Lemma \ref{lem:ODE}, Lemma \ref{lem:A priori bound} and
	Proposition~\ref{prop:concentration} that, for $N$ large enough and $t\geq \log(N)^3$,
	\begin{equation}
		\quad \dot{a}(t)
		\leq-\left(\frac{c_{\ref{lem:A priori bound}}}{tN^{c}}\right)^2\int_0^L
		(h(x))^2\tilde{h}(x)dx\leq- (c_{\ref{lem:A priori bound}})^2\frac{2\sigma^2 }{t^2}\frac{1}{N^c}.
	\end{equation}
	It then follows from an integration that for $\log(N)^3\leq t<s$,
	\begin{equation*}
		a(s)-a(t)\leq \frac{2\sigma^2(c_{\ref{lem:A priori bound}})^2}{N^c} \left(\frac{1}{s}-\frac{1}{t}\right).
	\end{equation*}
	Letting $s\to\infty$ and using the second part of Lemma~\ref{lem:ODE}
	give the result.
\end{proof}

\section{Proofs of Proposition~\ref{prop:kolmogorov} and Theorem~\ref{th:Kolmogorov}}

\newcommand{\ts}{{t_*}}

\begin{lemma}\label{lem:ub_f}
	Let $\ts=\log(N)^3$.
	For all $N$ large enough, all $t\geq 2\ts$ and all $x\in (0,L)$, we have
	\begin{equation}
		\label{eq:ub_f}
		f(t,x)= \left(1+O\left(N^{-1}\right)\right)a(t-\ts)h(x).
	\end{equation}
\end{lemma}
\begin{proof}
	We know from the branching property that, for all $N\geq N_0$, $t>\ts$ and
	$x\in(0,L)$,
	\begin{align}\label{eq:branching}
		f(t,x)= \P_x \left(\bigcup_{v \in \Nc_{\ts}}
		\{ Z^{(v)}({t-\ts})>0\}\right),
	\end{align}
	where $Z^{(v)}(t-\ts)$ denotes the number of descendants at time $t$
	of the particle  $v$ living at time $\ts$.
	A union bound combined with the many-to-one formula (Lemma~\ref{lem:many-to-one0})
	entails that
	\begin{align*}
		f(t,x)
		\leq \E_x \left[\sum_{v \in \Nc_{\ts}} \P_{x_v}(Z^{(v)}(t-\ts)>0)\right]
		= \int_0^L p_{\ts}(x,y)  f(t-\ts,y) dy.
	\end{align*}
	Then, note that for $N$ large enough and $t\geq 2\ts$, we have
	$\tfrac{\pi^2}{L^2}t^2\geq \log(N)$.
	The result thus follows from the first part of~\cref{lem:hk}.
\end{proof}

Before establishing the lower bound on $f$, we first show how \cref{lem:ub_f} can
be used to obtain a sharper upper bound on $a$.
\begin{lemma}
	\label{lem:uba}
	There exists a constant $c_{\ref{lem:uba}}>0$ such that for sufficiently large $N$ and $t\geq {\log(N)^6}$,
	\[
		a(t)\leq \frac{c_{\ref{lem:uba}}}{N^c\log(N)^2}.
	\]
\end{lemma}

\begin{proof} Let $1\ll A\ll L$ and write $a$ as
	\begin{equation}\label{eq:decomp_a}
		a(t)=\int_0^{A}f(t,x)\tilde{h}(x)dx
		+\int_{A}^\lcn f(t,x)\tilde{h}(x)dx.
	\end{equation}
	First, we see from Remark~\ref{rem:s_fittest} that, for sufficiently large $N$,
	we have
	\begin{equation}\label{eq:decomp_a1}
		\int_0^{A}f(t,x)\tilde{h}(x)dx\leq
		\int_0^{A}\tilde{h}(x)dx\leq \frac{2Ae^{\beta A}}{c^6N^c\log(N)^6}.
	\end{equation}
	On the other hand, Lemma~\ref{lem:ub_f} shows that, for sufficiently large $N$
	and all $t\geq 2\ts$,
	\begin{equation*}
		\int_A^Lf(t,x)\tilde{h}(x)dx=\left(1+O\left(N^{-1}\right)\right)a(t-\ts)
		\int_A^L\tilde{h}(x)h(x),
	\end{equation*}
	where $\ts$ is as in the statement of Lemma~\ref{lem:ub_f}.
	It then follows from~Lemma~\ref{lem:A priori bound}
	that, for sufficiently large $N$ and $t\geq \log(N)^6$,
	\[
		a(t-\ts)\leq a(t)+\frac{c_{\ref{lem:A priori bound}}}{(t-\tilde t)N^c}
		\leq a(t)+\frac{2c_{\ref{lem:A priori bound}}}{\log(N)^6N^c}.
	\]
	Moreover, an explicit calculation shows that for sufficiently large $N$,
	\begin{equation*}
		\int_A^L\tilde{h}(x)h(x)dx\leq 1-\left(\frac{A}{L}\right)^3,
	\end{equation*}
	(see \cref{rem_density} for more details).
	Since $\left(\tfrac{A}{L}\right)^3\gg N^{-1}$, the previous estimates imply that
	for sufficiently large $N$
	and all $t\geq \log(N)^6$,
	\begin{equation*}
		\int_A^Lf(t,x)\tilde{h}(x)dx\leq
		\left(1-\frac{A^3}{2L^3}\right)a(t)+\frac{2c_{\ref{lem:A priori bound}}}{\log(N)^6N^c}.
	\end{equation*}
	Combining this with \eqref{eq:decomp_a} and
	\eqref{eq:decomp_a1} shows that, for $N$ large enough and $t\geq \log(N)^6$,
	\begin{equation*}
		a(t)\leq
		C\frac{e^{\beta A}}{A^2}\frac{1}{N^c\log(N)^3}
	\end{equation*}
	We conclude the proof of the lemma by choosing $A$ such that
	$1\ll A\ll \log\log(N)$.
\end{proof}

\begin{proof}[Proof of Proposition~\ref{prop:kolmogorov}]
	\textit{Step 1: Upper bound.} The result is a consequence of Lemma~\ref{lem:ub_f}
	and Step 3 below.

	\textit{Step 2: Lower bound.} Using the alternative spinal decomposition~\eqref{chg:measure},
	Powell~\cite{powell19} shows that, for all $N\geq N_0$, all $t>\ts$ and $x\in(0,L)$,
	\begin{equation}
		\label{eq:lb1}
		\frac{f(t,x)}{h(x)}
		\geq \left(\int_0^L \frac{f(t-\ts,y)}{h(y)}q_\ts(x,y)dy\right)\times
		\left( 1-\mathfrak{Q}(x,\ts)\sup_{x\in (0,L)}\P_x(Z(t-\ts)>0)\right),
	\end{equation}
	where $\mathfrak{Q}(x,\ts)$ refers to the number of offspring produced by the spine particle started from $x$ before time $\ts$. Note that, by construction of the spine particle (see the discussion
	below Proposition~\ref{prop:kolmogorov}),
	\[
		\mathfrak{Q}(x,\ts)\leq C\ts,
	\]
	where $C$ is a constant that does not depend on $N$ nor on $x$ (the spine particles
	duplicate at constant rate $1$). Next,
	we see from a {standard coupling argument} combined with Lemma~\ref{lem:ub_f}
	that, for $N$ large enough and $t\geq 2\ts$,
	\[
		\sup_{x\in (0,L)}\P_x(Z(t-\ts)>0)\leq \P_0(Z(t-\ts)>0) \leq 2a(t-\ts)h(0).
	\]
	Moreover, by definition of $h$ (see \eqref{def:h}), we see that
	\[
		h(0)\leq C\frac{N^c}{\log(N)^3},
	\]
	for some constant $C$ is a constant that does not depend on $N$ (nor on $x$).
	This along with Lemma~\ref{lem:uba} yields that for all $N$ large enough,
	$t\geq 2\log(N)^6$ and $x\in(0,L)$,
	\[
		\mathfrak{Q}(x,\ts)\sup_{x\in (0,L)}\P_x(Z(t-\ts)>0)
		\leq \vep.
	\]
	We now bound the first factor in \eqref{eq:lb1}. Lemma~\ref{lem:hk} entails that
	for all $N$ large enough, $t\geq 2\log(N)^6$ and $x\in(0,L)$,
	\[
		\int_0^L \frac{f(t-\ts,y)}{h(y)}q_\ts(x,y)dy\leq (1-\vep)a(t-\ts).
	\]
	\textit{Step 3: Conclusion.} Putting all of this together, we have that for sufficiently large $N$, for all
	$t\geq 2\log(N)^6$ and $x\in (0,L)$,
	\begin{equation}
		\label{eq:boundf}
		\left|\frac{f(t,x)}{a(t-\ts)h(x)}-1\right|\leq \vep.
	\end{equation}
	\bigskip
	It now remains to prove that for $N$ large enough and $t\geq 2\log(N)^6$,
	\[
		\left|\frac{a(t)}{a(t-\ts)}-1\right|\leq \vep.
	\]
	Following~\cite{powell19}, we note that, by definition of $a$,
	\[
		\left|\frac{a(t)}{a(t-\ts)}-1\right|=\left|\int_0^L\left(\frac{f(t,x)}{h(x)a(t-\ts)}-1\right)
		h(x)\tilde{h}(x)dx\right|\leq \vep \int h\tilde h \leq \vep,
	\]
	for sufficiently large $N$ and $t\geq 2\log(N)^6$ by \eqref{eq:boundf}.

	As a consequence, we have that, for all $N$ large enough, $t\geq2\log(N)^6$ and $x\in (0,L)$
	\begin{equation}
		\label{eq:fine_estimate_f}
		(1-\vep)a(t)h(x)\leq f(t,x)\leq (1+\vep)a(t)h(x)
	\end{equation}
	so that, by Lemma~\ref{lem:ODE},
	\begin{equation*}
		\frac{1}{(1+\vep)\frac{\Sigma^2}{2}(t_2-t_1)+\frac{1}{a(t_1)}}\leq a(t_2)\leq
		\frac{1}{(1-\vep)\frac{\Sigma^2}{2}(t_2-t_1)+\frac{1}{a(t_1)}}.
	\end{equation*}
	For sufficiently large $N$, $t_2>t_1\geq2\log(N)^6$ and $x\in(0,L)$.
	We conclude the proof by remarking that
	for $t_2>\log(N)^7$ and $t_1=3\log(N)^6$, we have $\frac{t_2}{t_1}\to 0$
	and
	\[
		\lim_{N\to\infty}{\Sigma^2t_2 a(t_1)}=+\infty,
	\]
	by Lemma~\ref{lem:uba} and Proposition~\ref{prop:concentration}.
	The result then stems from \eqref{eq:fine_estimate_f}.
\end{proof}

\begin{proof}[Proof of Theorem~\ref{th:Kolmogorov}]
	The result follows from Proposition~\ref{prop:kolmogorov} by setting
	$t=sN^{1-c}$ for some $s>0$, and recalling from Proposition~\ref{prop:concentration}
	that $\Sigma^2/N^c\to\sigma^2$ as $N\to\infty$.
\end{proof}

\appendix

\section{Appendix: Concentration estimates}
\label{sec:proof_variance}

\begin{proof}[Proof of Proposition~\ref{prop:concentration}]
	Recall from \eqref{def:h} that
	\begin{equation}
		\Sigma(z)^2=\frac{4\gamma e^{\beta L}}{(L+\beta)^2}\int_0^z e^{-\beta x}v_1(x)^3dx=:\frac{4\gamma e^{\beta L}}{(L+\beta)^2} I_z. \label{Sigma:M:formula}
	\end{equation}
	First, we remark that for all $\alpha>0$ and $0\leq x_1\leq x_2\leq L$, we have
	\begin{equation} \label{eq:IPP}
		(\beta^2+\alpha^2)\int_{x_1}^{x_2}e^{\beta x}\sin(\alpha x)dx= \beta\left[e^{\beta {x_2}}\sin(\alpha {x_2})-e^{\beta {x_1}}\sin(\alpha {x_1})\right]-\alpha\left[e^{\beta {x_2}}\cos(\alpha {x_2})-e^{\beta {x_1}}\cos(\alpha {x_1})\right].
	\end{equation}
	This equality can be obtained by consecutive integrations by parts. We will also make use of the following identities:
	\begin{multline}\label{eq:identity}
		\forall x\in \mathbb{R}, \quad \sin^3(x)=\frac{1}{4}\left(3\sin(x)-\sin(3x)\right), \quad \cos^3(x)=\frac{1}{4}\left(3\cos(x)+\cos(3x)\right) \\
		\sin(\gamma L)=\gamma, \quad \cos{\gamma L}=-\beta, \quad \sin(3\gamma L)=3\gamma -4\gamma^3, \quad \cos(3\gamma L)=3\beta-4\beta^3, \quad \gamma^2+\beta^2=1,
	\end{multline}
	where the second and the third inequalities are established in Lemma~\ref{lem:decomp_spect}. Thus,
	\begin{equation*}
		I_z=\frac{3}{4}\underbrace{\int_0^ze^{-\beta x}\sin(\gamma(L-x))}_{\displaystyle =:I_{z,1}}dx-\frac{1}{4}\underbrace{\int_0^ze^{-\beta x}\sin(3\gamma(L-x))dx}_{\displaystyle =:I_{z,3}}.
	\end{equation*}
	It follows from a change a variable, \eqref{eq:IPP} and \eqref{eq:identity} that
	\begin{align*}
		I_{z,1} & =\beta\left(\sin(\gamma L)- e^{-\beta z}\sin(\gamma (L-z))\right)-\gamma \left(\cos(\gamma L)-e^{-\beta z}\cos(\gamma(L-z))\right) \\
		        & = 2\beta \gamma +e^{-\beta z}\left(\gamma \cos (\gamma (L-z))-\beta \sin(\gamma (L-z))\right),
	\end{align*}
	and that
	\begin{align}
		I_{z,3} & \label{eq:I3}
		=(1+8\gamma^2)^{-1}\left[\beta\left(\sin(3\gamma L)-e^{-\beta z}\sin(3\gamma (L-z))\right)-3\gamma\left(\cos(3\gamma L)-e^{-\beta z}\cos(3\gamma(L-z))\right)\right] \nonumber                         \\
		        & =(1+8\gamma^2)^{-1}\left[\beta \sin(3\gamma L)-3\gamma \cos(3\gamma L) +e^{-\beta z} (3\gamma\cos(3\gamma (L-z))-\beta \sin(3\gamma (L-z)))\right] \nonumber                                 \\
		        & =(1+8\gamma^2)^{-1}\underbrace{\left[6\beta \gamma -16\beta \gamma^3+e^{-\beta z} (3\gamma\cos(3\gamma (L-z))-\beta \sin(3\gamma (L-z)))\right]}_{\displaystyle :=\tilde{I}_{z,3}} \nonumber
	\end{align}
	In particular, for $z=L$, we have
	\[
		I_{L,1}=2\beta \gamma + \gamma e^{-\beta L}{=3\beta\gamma +o(\gamma^3)},
		\quad
		\text{and}
		\quad I_{L,3}=(1+8\gamma^2)^{-1}(6\beta\gamma-16\beta\gamma^3+3\gamma e^{-\beta L})
		{=6\beta\gamma-64\beta \gamma^3+o(\gamma^3)},
	\]
	so that, $I_L=\frac{1}{4}\left(64\gamma^3+o(\gamma^3)\right)$,
	which, together with \eqref{Sigma:M:formula}, \eqref{eq:equiv_L} and \eqref{eq:def_beta} yields \eqref{eq:equiv_variance} (with $\sigma^2=64\pi^2/c^6$).
\end{proof}

\begin{proof}[Proof of Proposition~\ref{prop:concentration2}]
	Factorising the previous estimates, we get
	\begin{align*}
		3I_{z,1}-\tilde{I}_{z,3} & =16\beta\gamma^3+e^{-\beta z}\left(3\gamma(\cos(\gamma(L-z)-\cos(3\gamma(L-z))))+\beta(\sin(3\gamma(L-z))-3\beta\sin(\gamma(L-z)))\right) \\
		                         & =16\beta\gamma^3+e^{-\beta z} \left(12\gamma\sin(\gamma(L-z))^2\cos(\gamma(L-z))-4\beta\sin(\gamma(L-z))^3\right)                         \\
		                         & =16\beta\gamma^3-4e^{-\beta z}\sin(\gamma(L-z))^2\left(\beta\sin(\gamma(L-z))-3\gamma\cos(\gamma(L-z))\right),
	\end{align*}
	where we used \eqref{eq:identity} and the identity $\cos(x)-\cos(y)=-2\sin((x-y)/2)\sin((x+y)/2)$ to get the second line.
	Let us now assume that $1\ll z\ll L$. Then, $\gamma \ll \gamma z \ll 1$ and a Taylor expansion then shows that
	\begin{align*}
		3I_{z,1}-\tilde{I}_{z,3}
		                         & ={16\beta\gamma^3-4e^{-\beta z}((\gamma z)^3
		+o((\gamma z)^3))},                                                                    \\
		-8\gamma^2\tilde I_{z,3} & =-48\beta\gamma^3+e^{-\beta z}(o((\gamma z)^3))+o(\gamma^4)
	\end{align*}
	so that
	\[
		3I_{z,1}-I_{z,3}=64\gamma^3 -4e^{-\beta z }((\gamma z)^3+o((\gamma z)^3))+o(\gamma^4).
	\]
	Applying this to $z= A=6\log\log N-\log\log\log(N)$  yields the first part of the proposition.

	We now compute the quantity
	\begin{equation*}
		J_z:=\int_0^z \tilde{h}(x)dx=\frac{1}{\gamma}\int_0^z e^{\beta(x-L)}\sin(\gamma(L-x))dx=\frac{1}{\gamma}\int_{L-z}^Le^{-\beta x}\sin(\gamma x)dx.
	\end{equation*}
	An integration by part shows that, for all $0\leq x_1\leq x_2\leq L$, we have
	\begin{equation*}
		(\beta^2+\alpha^2)\int_{x_1}^{x_2}e^{-\beta x}\sin(\alpha x)dx= -\beta  \left[e^{-\beta {x_2}}\sin(\alpha {x_2})-e^{-\beta {x_1}}\sin(\alpha {x_1})\right]-\alpha\left[e^{-\beta {x_2}}\cos(\alpha {x_2})-e^{-\beta {x_1}}\cos(\alpha {x_1})\right].
	\end{equation*}
	Putting this together with \eqref{eq:identity} shows that
	\begin{equation*}
		J_z=\frac{1}{\gamma}e^{-\beta(L-z)}\left(\beta\sin(\gamma(L-z))+\gamma\cos(\gamma(L-z))\right).
	\end{equation*}
	In particular, we see that $J_L=1$ (as mentioned in \eqref{eq:normalisation}).
	For $1\ll z\ll L$, a Taylor expansion shows that
	\begin{equation*}
		J_z=\frac{1}{\gamma}e^{-\beta(L-z)}(\gamma z+o(\gamma z)).
	\end{equation*}
	Plugging $z=A$ in the above equation shows that $N^cJ_A\to {6/c^6}$ as $N$ goes to $\infty$.
	This completes the proof.
\end{proof}

\begin{rem}\label{rem:s_fittest}
	The above calculations show that for all $1\ll A \ll L$,
	as $N\to\infty$,
	\begin{equation*}
		\frac{1}{N^c}\Sigma(A)^2\to\sigma^2,
		\quad \text{and} \quad
		N\int_0^A\tilde{h}(x)dx \sim \frac{N^{1-c}}{c^6\log(N)^6}  \left(Ae^{\beta A}\right).
	\end{equation*}
\end{rem}

\begin{rem}
	\label{rem_density}
	For all $1\ll A\ll L$
	\[
		\int_{0}^{A}h(x)\tilde{h}(x)dx\sim \frac{2\pi^2}{3}\frac{A^3}{L^3}.
	\]
	Indeed, a direct calculation shows that
	\[
		\int_0^A\sin(\gamma(L-x))^2dx=\frac{A}{2}-\frac{1}{4\gamma}\left(
		\sin(2\gamma L)-\sin(2\gamma(L-A))
		\right)=\frac{A}{2}-\frac{1}{2\gamma}\cos(\gamma (2L-A))\sin(\gamma A).
	\]
	Next, A Taylor expansion yields
	\[
		\cos(\gamma (2L-A))\sin(\gamma A)=\gamma A -\frac{2}{3}(\gamma A)^3
		+o((\gamma A)^3 ).
	\]
	Hence,
	\[
		\int_0^A\sin(\gamma(L-x))^2dx\sim \frac{\gamma^2 A^3}{3}
	\]
	and by Remark~\ref{rem:bij_L},
	\[
		\int_{0}^{A}h(x)\tilde{h}(x)dx\sim \frac{2}{L}\frac{\gamma^2 A^3}{3}\sim\frac{2\pi^2}{3}\frac{A^3}{L^3}.
	\]
\end{rem}
}

\pagebreak

\bibliographystyle{plain}
\bibliography{mullers_ratchet.bib}

\begin{thebibliography}{10}

\bibitem{Barton:2017aa}
N.~H. Barton, A.~M. Etheridge, and A.~V\'eber.
\newblock The infinitesimal model: Definition, derivation, and implications.
\newblock {\em Theor. Popul. Biol.}, 118:50--73, 2017.

\bibitem{berestycki13}
J.~Berestycki, N.~Berestycki, and J.~Schweinsberg.
\newblock The genealogy of branching {B}rownian motion with absorption.
\newblock {\em Ann. Probab.}, 41(2):527--618, 2013.

\bibitem{birkner2005alpha}
M.~Birkner, J.~Blath, M.~Capaldo, A.~Etheridge, M.~M\"ohle, J.~Schweinsberg, and A.~Wakolbinger.
\newblock Alpha-stable branching and beta-coalescents.
\newblock {\em Electron. J. Probab.}, 10:303--325, 2005.

\bibitem{boenkost2022genealogy}
F.~Boenkost, F.~Foutel-Rodier, and E.~Schertzer.
\newblock The genealogy of nearly critical branching processes in varying environment.
\newblock {\em arXiv preprint arXiv:2207.11612}, 2022.

\bibitem{Bolthausen1998}
E.~Bolthausen and A.~S. Sznitman.
\newblock On {R}uelle's probability cascades and an abstract cavity method.
\newblock {\em Commun. Math. Phys.}, 197(2):247--276, 1998.

\bibitem{borodin2015handbook}
A.~N. Borodin and P.~Salminen.
\newblock {\em Handbook of Brownian motion: facts and formulae}.
\newblock Springer, 2015.

\bibitem{buffalo}
V.~Buffalo.
\newblock Quantifying the relationship between genetic diversity and population size suggests natural selection cannot explain lewontin's paradox.
\newblock {\em eLife}, 10:e67509, aug 2021.

\bibitem{Casanova2022}
A.~G. Casanova, S.~Smadi, and A.~Wakolbinger.
\newblock Quasi-equilibria and click times for a variant of {M}uller’s ratchet.
\newblock {\em Electron. J. Probab.}, 28, 2023.

\bibitem{charlesworth2009effective}
B.~Charlesworth.
\newblock Effective population size and patterns of molecular evolution and variation.
\newblock {\em Nat. Rev. Genet.}, 10(3):195--205, 2009.

\bibitem{Charlesworth93}
M.~T. Charlesworth and D.~Charlesworth.
\newblock The effect of deleterious mutations on neutral molecular variation.
\newblock {\em Genetics}, 134(4):1289--1303, 1993.

\bibitem{cvijovic2018effect}
I.~Cvijovi{\'c}, B.~H. Good, and M.~M. Desai.
\newblock The effect of strong purifying selection on genetic diversity.
\newblock {\em Genetics}, 209(4):1235--1278, 2018.

\bibitem{depperschmidt2019treevalued}
A.~Depperschmidt and A.~Greven.
\newblock Tree valued {F}eller diffusion.
\newblock {\em arXiv preprint arXiv:1904.02044}, 2019.

\bibitem{depperschmidt_marked_2011}
A.~Depperschmidt, A.~Greven, and P.~Pfaffelhuber.
\newblock Marked metric measure spaces.
\newblock {\em Electron. Commun. Probab.}, 16:174--188, 2011.

\bibitem{Desai:2007aa}
M.~M. Desai and D.~S. Fisher.
\newblock Beneficial mutation--selection balance and the effect of linkage on positive selection.
\newblock {\em Genetics}, 176(3):1759--1798, 2007.

\bibitem{durrett_probability_2019}
R.~Durrett.
\newblock {\em Probability: Theory and Examples}.
\newblock Cambridge Ser. Stat. Probab. Math. Cambridge, fifth edition, 2019.

\bibitem{durrett2011}
R.~Durrett and J.~Mayberry.
\newblock Traveling waves of selective sweeps.
\newblock {\em Ann. Appl. Probab.}, 21(2):699--744, 2011.

\bibitem{etheridge_2011}
A.~M. Etheridge.
\newblock {\em Some Mathematical Models from Population Genetics}.
\newblock Lecture Notes Math. Springer, 2011.

\bibitem{Etheridge2009}
A.~M. Etheridge, P.~Pfaffelhuber, and A.~Wakolbinger.
\newblock How often does the ratchet click? facts, heuristics, asymptotics.
\newblock In {\em Trends Stoch. Anal.}, volume 353 of {\em London Math. Soc. Lecture Note Ser.}, pages 365--390. Cambridge Univ. Press, 2009.

\bibitem{forien2025stochastic}
R.~Forien, E.~Schertzer, Z.~Talyig\'as, and J.~Tourniaire.
\newblock Stochastic neutral fractions and the effective population size.
\newblock {\em arXiv preprint arXiv:2502.05062}, 2025.

\bibitem{foutel22}
F.~Foutel-Rodier and E.~Schertzer.
\newblock Convergence of genealogies through spinal decomposition with an application to population genetics.
\newblock {\em Probab. Theory Relat. Fields}, 187(3-4):697--751, 2023.

\bibitem{foutel2024convergence}
F.~Foutel-Rodier, E.~Schertzer, and J.~Tourniaire.
\newblock Convergence of spatial branching processes to $\alpha$-stable {CSBP}s: Genealogy of semi-pushed fronts.
\newblock {\em arXiv preprint arXiv:2402.05096}, 2024.

\bibitem{greven2009convergence}
A.~Greven, P.~Pfaffelhuber, and A.~Winter.
\newblock Convergence in distribution of random metric measure spaces ({$\Lambda$}-coalescent measure trees).
\newblock {\em Probab. Theory Relat. Fields}, 145(1):285--322, 2009.

\bibitem{Haigh1978}
J.~Haigh.
\newblock The accumulation of deleterious genes in a population -- {M}uller's ratchet.
\newblock {\em Theor. Popul. Biol.}, 14(2):251--267, oct 1978.

\bibitem{harris21}
S.~C. Harris, E.~Horton, A.~E. Kyprianou, and M.~Wang.
\newblock Yaglom limit for critical nonlocal branching {M}arkov processes.
\newblock {\em Ann. Probab.}, 50(6), 2022.

\bibitem{harris2017spine}
S.~C. Harris and M.~I. Roberts.
\newblock The many-to-few lemma and multiple spines.
\newblock {\em Ann. Inst. Henri Poincaré, Probab. Stat.}, 53:226--242, 2017.

\bibitem{Igelbrink2023}
J.L. Igelbrink, A.~González~Casanova, C.~Smadi, and A.~Wakolbinger.
\newblock Muller’s ratchet in a near-critical regime: {T}ournament versus fitness proportional selection.
\newblock {\em Theor. Popul. Biol.}, 158:121--138, 2024.

\bibitem{BKT2023+}
K.~A. Khudiakova, F.~Boenkost, and J.~Tourniaire.
\newblock Genealogies under purifying selection.
\newblock bioRxiv preprint bioRxiv 2024.10.15.618444, 2024.

\bibitem{Kingman1982}
J.~F.~C. Kingman.
\newblock On the genealogy of large populations.
\newblock {\em J. Appl. Probab.}, (Special Vol. 19A):27--43, 1982.

\bibitem{Lawler:2018vn}
G.~F. Lawler.
\newblock {\em Introduction to stochastic processes}.
\newblock Chapman Hall/CRC, 2018.

\bibitem{Lewontin1974}
R.~C. Lewontin.
\newblock {\em The genetic basis of evolutionary change}.
\newblock Number~25 in @Columbia biological series. Columbia Univ. Pr., New York [u.a.], 1974.

\bibitem{lyons1995conceptual}
R.~Lyons, R.~Pemantle, and Y.~Peres.
\newblock Conceptual proofs of l log l criteria for mean behavior of branching processes.
\newblock {\em Ann. Probab.}, pages 1125--1138, 1995.

\bibitem{Moehle1998}
M.~M\"ohle.
\newblock Robustness results for the coalescent.
\newblock {\em J. Appl. Probab.}, 35(2):438--447, 1998.

\bibitem{Muller1964}
H.~J. Muller.
\newblock The relation of recombination to mutational advance.
\newblock {\em Mutat. Res.}, 1(1):2--9, 1964.

\bibitem{neher2013genealogies}
R.~A. Neher and O.~Hallatschek.
\newblock Genealogies of rapidly adapting populations.
\newblock {\em Proc. Natl. Acad. Sci. USA}, 110(2):437--442, 2013.

\bibitem{neveu}
J.~Neveu.
\newblock A continuous-state branching process in relation with the {GREM} model of spin glass theory.
\newblock Rapport interne 267, Ecole polytechnique, 1992.

\bibitem{Pfaffelhuber_2012}
P.~Pfaffelhuber, P.~R. Staab, and A.~Wakolbinger.
\newblock Muller's ratchet with compensatory mutations.
\newblock {\em Ann. Appl. Probab.}, 22(5), oct 2012.

\bibitem{pinsky1995positive}
R.~G. Pinsky.
\newblock {\em Positive harmonic functions and diffusion}, volume~45.
\newblock Cambridge Univ. Press, 1995.

\bibitem{popovic2004}
L.~Popovic.
\newblock Asymptotic genealogy of a critical branching process.
\newblock {\em Ann. Appl. Probab.}, 14:2120--2148, 2004.

\bibitem{powell19}
E.~Powell.
\newblock An invariance principle for branching diffusions in bounded domains.
\newblock {\em Probab. Theory Relat. Fields}, 173(3):999--1062, 2019.

\bibitem{roberts2021gaussian}
M.~I. Roberts and J.~Schweinsberg.
\newblock A {G}aussian particle distribution for branching {B}rownian motion with an inhomogeneous branching rate.
\newblock {\em Electron. J. Probab.}, 26:1--76, 2021.

\bibitem{tourniaire2023tree}
E.~Schertzer and J.~Tourniaire.
\newblock Spectral analysis and k-spine decomposition of inhomogeneous branching brownian motions. genealogies in fully pushed fronts.
\newblock {\em Ann. Probab.}, 53(4):1382--1433, 2025.

\bibitem{schweinsberg2017rigorous}
J.~Schweinsberg.
\newblock Rigorous results for a population model with selection {I}: evolution of the fitness distribution.
\newblock {\em Electron. J. Probab.}, 22, 2017.

\bibitem{tourniaire2021branching}
J.~Tourniaire.
\newblock A branching particle system as a model of semipushed fronts.
\newblock {\em Ann. Probab.}, 52(6):2106--2172, 2024.

\bibitem{Walczak2012}
A.~M. Walczak, L.~E. Nicolaisen, J.~B. Plotkin, and M.~M. Desai.
\newblock The structure of genealogies in the presence of purifying selection: A fitness-class coalescent.
\newblock {\em Genetics}, 190(2):753--779, feb 2012.

\bibitem{wright1931evolution}
S.~Wright.
\newblock Evolution in {M}endelian populations.
\newblock {\em Genetics}, 16(2):97--159, 1931.

\bibitem{zettl10}
A.~Zettl.
\newblock {\em Sturm-Liouville theory}.
\newblock Amer. Math. Soc., 2012.

\end{thebibliography}
\end{document}